\documentclass[a4paper, 11pt]{article}
\usepackage[section]{algorithm}
\usepackage[noend]{algpseudocode}
\usepackage{amsmath, amssymb, amsthm}
\usepackage{chngcntr}
\usepackage{enumitem}
\usepackage[left=70pt, right=70pt, top=80pt, bottom=90pt]{geometry}
\usepackage{graphics}
\usepackage{mathtools}
\usepackage[numbers, sort&compress]{natbib}
\usepackage{nicematrix}
\usepackage[labelformat=simple]{subcaption}
\usepackage{tikz}
\usetikzlibrary{arrows.meta,calc,fit,graphs.standard,matrix,positioning,shapes.geometric,quotes}
\usepackage{url}
\usepackage{xparse}
\usepackage[hidelinks, hypertexnames=false]{hyperref}
\usepackage[capitalise, nameinlink, noabbrev]{cleveref}

\algrenewcommand{\algorithmicrequire}{\textbf{Input:}}
\algrenewcommand{\algorithmicensure}{\textbf{Output:}}
\NewDocumentCommand{\algorithmicor}{}{\ \textbf{or}\ }
\NewDocumentCommand{\IfThen}{m m}{\algorithmicif\ #1\ \algorithmicthen\ #2}

\DeclareMathOperator*{\argmax}{arg\,max}
\DeclareMathOperator{\finitefield}{GF}
\DeclareMathOperator{\order}{\mathit{O}}
\DeclareMathOperator{\rank}{rank}
\DeclareMathOperator{\sequence}{\mathit{S}}
\DeclareMathOperator{\confun}{\mathit{\lambda}}
\DeclareMathOperator{\galebasis}{gb}
\DeclareMathOperator{\pathwidth}{pw}
\DeclareMathOperator{\rankm}{\mathit{r}}
\DeclareMathOperator{\childz}{\textsf{child}_0}
\DeclareMathOperator{\childo}{\textsf{child}_1}
\DeclareMathOperator{\labeldd}{\textsf{label}}
\DeclareMathOperator{\rootdd}{\textsf{root}}

\newtheorem{theorem}{Theorem}[section]
\newtheorem{corollary}[theorem]{Corollary}
\newtheorem{lemma}[theorem]{Lemma}
\newtheorem{proposition}[theorem]{Proposition}
\theoremstyle{definition}
\newtheorem{example}[theorem]{Example}
\newtheorem{remark}[theorem]{Remark}
\NewDocumentEnvironment{myexample}{o}{\IfValueTF{#1}{\begin{example}[#1]}{\begin{example}}}{\qed\end{example}}
\NewDocumentEnvironment{myremark}{o}{\IfValueTF{#1}{\begin{remark}[#1]}{\begin{remark}}}{\qed\end{remark}}

\counterwithin{figure}{section}

\crefname{enumi}{}{}
\Crefname{enumi}{}{}

\setlist[enumerate]{align=left, leftmargin=*, noitemsep}
\setlist[itemize]{leftmargin=*, noitemsep}

\DeclarePairedDelimiterX{\set}[2]{\lbrace}{\rbrace}{#1\mathrel{:}#2}
\DeclarePairedDelimiterXPP{\fun}[2]{#1}{\lparen}{\rparen}{}{#2}
\DeclarePairedDelimiter{\pqty}{\lparen}{\rparen}
\DeclarePairedDelimiter{\bqty}{\lbrack}{\rbrack}
\DeclarePairedDelimiter{\Bqty}{\lbrace}{\rbrace}
\DeclarePairedDelimiter{\abs}{\lvert}{\rvert}

\NewDocumentCommand{\edgelabel}{m}{{\footnotesize #1}}
\tikzset{
  ddnode/.style={circle, draw, minimum width=0.5cm},
  reducearrow/.style={-Stealth, line width=1mm},
  terminal/.style={draw, rectangle},
  triangle/.style={anchor=apex, draw, isosceles triangle, minimum width=1cm, shape border rotate=90},
  zedge/.style={-Latex, dashed},
  oedge/.style={-Latex}
}
\def\ddsep{1.3}

\NewDocumentCommand{\qbinom}{m m}{\genfrac{[}{]}{0pt}{}{#1}{#2}}
\NewDocumentCommand{\galeorder}{m}{#1_{\mathrm{G}}}
\NewDocumentCommand{\gset}{}{E}
\NewDocumentCommand{\ei}{O{\preceq}}{\gset_{#1, i}}
\NewDocumentCommand{\eivar}{O{\preceq}}{\gset \setminus \ei[#1]}
\NewDocumentCommand{\leftjustified}{s}{\IfBooleanTF{#1}{\vartriangleleft}{\trianglelefteq}}
\NewDocumentCommand{\contraction}{m m}{#1 \mathrel{/} #2}
\NewDocumentCommand{\deletion}{m m}{#1 \setminus #2}
\NewDocumentCommand{\minor}{m m m}{\contraction{\deletion{#1}{#2}}{#3}}
\NewDocumentCommand{\minori}{m}{\minor{M}{\pqty{\ei \setminus #1}}{#1}}
\NewDocumentCommand{\restrict}{m m}{#1 | #2}
\NewDocumentCommand{\bases}{}{\mathcal{B}}
\NewDocumentCommand{\isets}{}{\mathcal{I}}
\NewDocumentCommand{\bdd}{s m o o}{\IfBooleanTF{#1}{\fun*}{\fun}{\textsf{B}}{#2\IfValueT{#3}{,\pqty{#3,#4}}}}
\NewDocumentCommand{\zdd}{s m o o}{\IfBooleanTF{#1}{\fun*}{\fun}{\textsf{Z}}{#2\IfValueT{#3}{,\pqty{#3,#4}}}}
\NewDocumentCommand{\labelroot}{s m}{\IfBooleanTF{#1}{\labeldd\pqty*{\rootdd\pqty*{#2}}}{\labeldd\pqty{\rootdd\pqty{#2}}}}

\title{On the sizes of BDDs and ZDDs representing matroids}
\author{Hiromi Emoto\thanks{Graduate School of Informatics, Kyoto University}\textsuperscript{\phantom{*},}\thanks{E-mail: \texttt{emoto.hiromi.63w@kyoto-u.jp}} \and Yuni Iwamasa\footnotemark[1]\textsuperscript{\phantom{*},}\thanks{E-mail: \texttt{iwamasa@i.kyoto-u.ac.jp}} \and Shin-ichi Minato\footnotemark[1]\textsuperscript{\phantom{*},}\thanks{E-mail: \texttt{minato@i.kyoto-u.ac.jp}}}
\date{\today}

\begin{document}
\maketitle

\begin{abstract}
  Matroids are often represented as oracles since there are no unified and compact representations for general matroids.
  This paper initiates the study of binary decision diagrams (BDDs) and zero-suppressed binary decision diagrams (ZDDs) as relatively compact data structures for representing matroids in a computer.
  This study particularly focuses on the sizes of BDDs and ZDDs representing matroids.
  First, we compare the sizes of different variations of BDDs and ZDDs for a matroid.
  These comparisons involve concise transformations between specific decision diagrams.
  Second, we provide upper bounds on the size of BDDs and ZDDs for several classes of matroids.
  These bounds are closely related to the number of minors of the matroid and depend only on the connectivity function or pathwidth of the matroid, which deeply relates to the classes of matroids called strongly pigeonhole classes.
  In essence, these results indicate upper bounds on the number of minors for specific classes of matroids and new strongly pigeonhole classes.
\end{abstract}
\begin{quote}
	{\bf Keywords: }
 Binary decision diagram, Zero-suppressed binary decision diagram, Matroid
\end{quote}

\section{Introduction}\label{sec:introduction}

\emph{Matroids}~\cite{b2d08917-c364-3337-af1c-57f8ed647c1c,23e983ed-3432-3b97-b3be-0381d88d5f62,97ddae92-4833-382b-b222-4cea46100541,37dbf464-6053-35fb-a8a6-c5f70977014e} are combinatorial structures that abstract linear independence in vector spaces.
This concept appears in various fields of mathematics and theoretical computer science such as graph theory, geometry, and combinatorial optimization~\cite{10.1093/acprof:oso/9780198566946.001.0001}.
In combinatorial optimization, for instance, matroids appear as the discrete structure behind several combinatorial optimization problems solvable in a greedy manner.

The formal definition of matroids is as follows.
Let \(\gset\) be a finite set and \(\isets \subseteq 2^{\gset}\) be a subset family of \(\gset\).
The pair \(\pqty{\gset, \isets}\) is called a \emph{matroid} if the following conditions \cref{item:i1,item:i2,item:i3} hold:
\begin{enumerate}[label=\textbf{(I\arabic*)}, ref=(I\arabic*)]
  \item\label{item:i1} \(\emptyset \in \isets\).
  \item\label{item:i2} If \(I \in \isets\) and \(I^\prime \subseteq I\), then \(I^\prime \in \isets\).
  \item\label{item:i3} For \(I_1, I_2 \in \isets\) with \(\abs{I_1} < \abs{I_2}\), there exists \(e \in I_2 \setminus I_1\) such that \(I_1 \cup \Bqty{e} \in \isets\).
\end{enumerate}
In this context, \(\gset\) is called the \emph{ground set} of \(M\), a member of \(\isets\) is called an \emph{independent set} of \(M\), and a maximal independent set is called a \emph{basis} of \(M\).

By \cref{item:i2}, the bases contain sufficient information to represent \(\isets\).
Indeed, the collection \(\bases\) of bases can characterize a matroid;
a subset family \(\bases \subseteq 2^{\gset}\) of the ground set \(\gset\) is the collection of bases of some matroid if and only if the following conditions \cref{item:b1,item:b2} hold:
\begin{enumerate}[label=\textbf{(B\arabic*)}, ref=(B\arabic*), series=base]
  \item\label{item:b1} \(\bases \neq \emptyset\).
  \item\label{item:b2} For \(B_1, B_2 \in \bases\) and \(x \in B_1 \setminus B_2\), there exists \(y \in B_2 \setminus B_1\) such that \(\pqty{B_1 \setminus \Bqty{x}} \cup \Bqty{y} \in \bases\).
\end{enumerate}
Hence the pair \(\pqty{\gset, \bases}\) is also referred to as a matroid.
For a matroid \(M\), we denote by \(\fun{\gset}{M}\), \(\fun{\isets}{M}\), and \(\fun{\bases}{M}\) the ground set, the collection of independent sets, and that of bases of \(M\), respectively.

The following are examples of matroids.

\begin{figure}[tb]
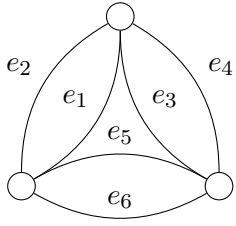

  \centering
  \tikz \graph {
    subgraph I_n [counterclockwise, empty nodes, nodes={draw, circle}, radius=1.5cm, V={a,b,c}];
    a --[bend left, "\(e_1\)"'] b;
    a --[bend right, "\(e_2\)"'] b;
    a --[bend right, "\(e_3\)"] c;
    a --[bend left, "\(e_4\)"] c;
    b --[bend left, "\(e_5\)"] c;
    b --[bend right, "\(e_6\)"] c;
  };
  \caption{An undirected graph \(G\)}
  \label{fig:transversal_not_laminar}
\end{figure}

\begin{myexample}[Cycle matroid]\label{eg:cycle}
  For an undirected graph \(G\), let \(\fun{V}{G}\) and \(\fun{E}{G}\) denote the vertex and edge sets of \(G\), respectively.
  Let \(\fun{\isets}{G}\) denote the family of edge subsets \(F \subseteq \fun{E}{G}\) such that the subgraph \(\pqty{\fun{V}{G}, F}\) of \(G\) is a forest.
  Then \(\pqty{\fun{E}{G}, \fun{\isets}{G}}\) forms a matroid, called the \emph{cycle matroid} \(\fun{M}{G}\) of \(G\).
  In this case, the collection \(\fun{\bases}{\fun{M}{G}}\) of bases of \(\fun{M}{G}\) is the family of edge subsets \(F \subseteq \fun{E}{G}\) such that \(F \in \fun{\isets}{G}\) and \(\abs{F} = \abs{\fun{V}{G}} - \omega\pqty{G}\), where \(\omega\pqty{G}\) is the number of connected components of \(G\).
  For example, we have \(\fun{E}{G} = \Bqty{e_1, \ldots, e_6}\), \(\fun{\bases}{\fun{M}{G}} = \set{B \subseteq \fun{E}{G}}{\text{\(\abs{B} = 2\), and \(B\) is not \(\Bqty{e_1, e_2}\), \(\Bqty{e_3, e_4}\), or \(\Bqty{e_5, e_6}\)}}\), and \(\fun{\isets}{G} = \fun{\bases}{\fun{M}{G}} \cup \Bqty{\emptyset, \Bqty{e_1}, \ldots, \Bqty{e_6}}\) for the undirected graph \(G\) in \Cref{fig:transversal_not_laminar}.
\end{myexample}

\begin{example}[\(\mathbb{F}\)-representable matroid]\label{eg:vector}
  Let \(A\) be an \(m \times \abs{\gset}\) matrix over a field \(\mathbb{F}\) with the set \(\gset\) of column labels and \(\isets\) be the family of subsets \(I \subseteq \gset\) such that the set of columns labeled by \(I\) in \(A\) is linearly independent in \(\mathbb{F}^m\).
  Then \(\pqty{\gset, \isets}\) forms a matroid, called the \emph{vector matroid} \(M \bqty{A}\) of \(A\).
  A matroid \(M\) is said to be \emph{\(\mathbb{F}\)-representable} if \(M = M \bqty{A}\) for some matrix \(A\) over \(\mathbb{F}\).
  In this case, \(\fun{\bases}{M \bqty{A}}\) is the family of subsets \(B \subseteq \gset\) such that \(B \in \isets\) and \(\abs{B} = \rank{A}\).
  For example, let \(\gset \coloneqq \Bqty{e_1, e_2, e_3, e_4}\), \(\bases \coloneqq \Bqty{\Bqty{e_1, e_2}, \Bqty{e_1, e_4}, \Bqty{e_2, e_4}}\), and \(\isets \coloneqq \bases \cup \Bqty{\emptyset, \Bqty{e_1}, \Bqty{e_2}, \Bqty{e_4}}\).
  Then \(\pqty{\gset, \isets}\) and \(\pqty{\gset, \bases}\) are \(\finitefield\pqty{2}\)-representable, where \(\finitefield\pqty{2}\) is the finite field of order \(2\).
  Indeed, they coincide with the vector matroid arising from the following matrix over \(\finitefield\pqty{2}\):
  \[
    \pushQED{\qed}
    \begin{pNiceMatrix}[first-row]
      e_1 & e_2 & e_3 & e_4 \\
      1 & 0 & 0 & 1 \\
      0 & 1 & 0 & 1
    \end{pNiceMatrix}\text{.}\qedhere
    \popQED
  \]
\end{example}

Several classes of matroids have polynomial representations in the sizes of their ground sets (see \cref{eg:cycle,eg:vector}).
However, since the number \(m_n\) of matroids whose size of the ground set is \(n\) satisfies that \(\log_2{\log_2{m_n}} \geq n - \frac{3}{2} \log_2{n} + \frac{1}{2} \log_2{\frac{2}{\pi}} - \fun{o}{1}\)~\cite{Bansal2015}, there are no unified and polynomial-size representations for general matroids.
In the field of combinatorial optimization, matroids are often represented as oracles as an indirect approach (see \cite{robinson_welsh_1980,Hausmann1981}).
However, when we actually solve a combinatorial optimization problem related to a matroid in a computer, we need to implement the oracle arising from the matroid.

In this paper, we initiate the study of \emph{binary decision diagrams} (\emph{BDDs})~\cite{10.1109/TC.1986.1676819} and \emph{zero-suppressed binary decision diagrams} (\emph{ZDDs})~\cite{10.1145/157485.164890} as relatively compact data structures for representing matroids uniformly.

BDDs and ZDDs are special types of directed acyclic graphs, which are obtained from binary decision trees by applying certain reductions (see \cref{sec:decision_diagrams} for the formal definition).
Informally, nodes in BDDs/ZDDs are labeled with the elements of the ground set, and a directed path in the graph represents a subset family of the ground set.
The size of a BDD/ZDD means the number of nodes in the graph.
The smaller its size, the better the compression.

In constructing a BDD/ZDD for a subset family of \(E\), we are required to fix a total order on \(E\), and the size of the BDD/ZDD heavily depends on the fixed total order.
For a totally ordered set \(\pqty{E, \preceq}\) and a set family \(\mathcal{S} \subseteq 2^{\gset}\), we denote by \(\bdd{\mathcal{S}}[E][\preceq]\) (resp. \(\zdd{\mathcal{S}}[E][\preceq]\)) the BDD (resp. ZDD) for \(\mathcal{S}\) with respect to \(\pqty{E, \preceq}\).
In some contexts, \(\bdd{\mathcal{S}}[E][\preceq]\) (resp. \(\zdd{\mathcal{S}}[E][\preceq]\)) itself means the family \(\mathcal{S}\).
We often abbreviate it to \(\bdd{\mathcal{S}}\) (resp. \(\zdd{\mathcal{S}}\)).
This study focuses on various properties, especially the size, of BDDs/ZDDs representing a collection of independent sets or bases of a matroid.

\subsection*{Our contribution}
\addcontentsline{toc}{subsection}{Our contribution}

In \cref{sec:structural_relations}, we compare different variations of BDDs/ZDDs for a matroid.
Since BDDs and ZDDs have different reduction rules, their structures can differ even when they represent the same set family.
In addition, as mentioned above, a single matroid \(M\) can be represented through various families such as the collection of independent sets \(\fun{\isets}{M}\) and bases \(\fun{\bases}{M}\).
The structure of the decision diagram can differ depending on the choices.

In this paper, we show the following relations on the sizes of the decision diagrams.
Here \(M^\ast\) denotes the \emph{dual} matroid of \(M\), whose ground set is \(\fun{\gset}{M}\) and whose collection of bases is \(\set{\fun{\gset}{M} \setminus B}{B \in \fun{\bases}{M}}\).

\begin{theorem}\label{thm:intro_structure}
  Let \(M\) be a matroid on the ground set \(\gset\) and \(\preceq\) be a total order on \(\gset\).
  Then the following hold:
  \begin{itemize}
    \item The sizes of both the ZDD \(\zdd{\fun{\bases}{M}}\) and the BDD \(\bdd{\fun{\isets}{M}}\) are never greater than that of the BDD \(\bdd{\fun{\bases}{M}}\) (proved by \cref{thm:clutter_zdd_and_clutter_bdd,thm:isets_and_bases}).
    \item The BDDs \(\bdd{\fun{\bases}{M}}\) and \(\bdd{\fun{\bases}{M^\ast}}\) are of equal size (proved by \cref{thm:clutter_bdd}).
    \item The ZDDs \(\zdd{\fun{\isets}{M}}\) and \(\zdd{\fun{\bases}{M}}\) and the BDD \(\bdd{\fun{\isets}{M^\ast}}\) all have the same size (proved by \cref{thm:isets_and_bases,thm:bases_zdd_and_dual_isets_bdd}).
  \end{itemize}
\end{theorem}

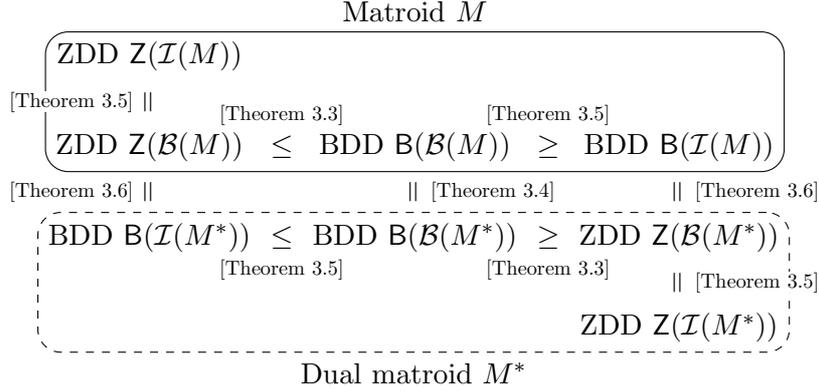
\begin{figure}[tb]
  \centering
  \begin{tikzpicture}
    \matrix [matrix of nodes] {
      |(zdd_i)| ZDD \(\zdd{\fun{\isets}{M}}\) & & & & \\
      |(zdd_i-to-zdd_b)| \rotatebox[origin=c]{90}{\(=\)} & & & & \\
      |(zdd_b)| ZDD \(\zdd{\fun{\bases}{M}}\) & |(zdd_b-to-bdd_b)| \(\leq\) & |(bdd_b)| BDD \(\bdd{\fun{\bases}{M}}\) & |(bdd_b-to-bdd_i)| \(\geq\) & |(bdd_i)| BDD \(\bdd{\fun{\isets}{M}}\) \\
      |(zdd_b-to-bdd_dual_i)| \rotatebox[origin=c]{90}{\(=\)} & & |(bdd_b-to-bdd_dual_b)| \rotatebox[origin=c]{90}{\(=\)} & & |(bdd_i-to-zdd_dual_b)| \rotatebox[origin=c]{90}{\(=\)} \\
      |(bdd_dual_i)| BDD \(\bdd{\fun{\isets}{M^\ast}}\) & |(bdd_dual_i-to-bdd_dual_b)| \(\leq\) & |(bdd_dual_b)| BDD \(\bdd{\fun{\bases}{M^\ast}}\) & |(bdd_dual_b-to-zdd_dual_b)| \(\geq\) & |(zdd_dual_b)| ZDD \(\zdd{\fun{\bases}{M^\ast}}\) \\
      & & & & |(zdd_dual_i-to-zdd_dual_b)| \rotatebox[origin=c]{90}{\(=\)} \\
      & & & & |(zdd_dual_i)| ZDD \(\zdd{\fun{\isets}{M^\ast}}\) \\
    };
    \node [draw, fit=(bdd_i)(bdd_b)(zdd_i)(zdd_b), inner sep=0cm, rounded corners=0.3cm] (matroid) {};
    \node [anchor=south] at (matroid.north) {Matroid \(M\)};
    \node [draw, dashed, fit=(zdd_dual_b)(zdd_dual_i)(bdd_dual_b)(bdd_dual_i), inner sep=0cm, rounded corners=0.3cm] (dual_matroid) {};
    \node [anchor=north] at (dual_matroid.south) {Dual matroid \(M^\ast\)};
    \begin{scope}[every node/.style={inner sep=0cm, scale=0.7}]
      \node [anchor=east, fill=white] at (zdd_i-to-zdd_b.west) {[\Cref{thm:isets_and_bases}]};
      \node [anchor=south] at (zdd_b-to-bdd_b.north) {[\Cref{thm:clutter_zdd_and_clutter_bdd}]};
      \node [anchor=south] at (bdd_b-to-bdd_i.north) {[\Cref{thm:isets_and_bases}]};
      \node [anchor=east] at (zdd_b-to-bdd_dual_i.west) {[\Cref{thm:bases_zdd_and_dual_isets_bdd}]};
      \node [anchor=west] at (bdd_b-to-bdd_dual_b.east) {[\Cref{thm:clutter_bdd}]};
      \node [anchor=west] at (bdd_i-to-zdd_dual_b.east) {[\Cref{thm:bases_zdd_and_dual_isets_bdd}]};
      \node [anchor=north] at (bdd_dual_i-to-bdd_dual_b.south) {[\Cref{thm:isets_and_bases}]};
      \node [anchor=north] at (bdd_dual_b-to-zdd_dual_b.south) {[\Cref{thm:clutter_zdd_and_clutter_bdd}]};
      \node [anchor=west, fill=white] at (zdd_dual_i-to-zdd_dual_b.east) {[\Cref{thm:isets_and_bases}]};
    \end{scope}
  \end{tikzpicture}
  \caption{Relations of the sizes between different decision diagrams}
  \label{fig:dd_repr_matroid}
\end{figure}

The results in \cref{thm:intro_structure} are summarized in \cref{fig:dd_repr_matroid}.
The proof of this theorem involves concise transformations between specific decision diagrams.
In this paper, the four types of BDDs/ZDDs representing either a collection of independent sets or bases of a matroid will be referred to as BDDs/ZDDs for a matroid.

In \cref{sec:width}, we provide upper bounds on the size (or the \emph{width}) of BDDs/ZDDs in terms of the value of the \emph{connectivity function} for several classes of matroids.
Now let \(E\) be a set \(\Bqty{e_1, \ldots, e_n}\) and \(\preceq\) be a total order on \(E\) such that \(e_1 \prec \cdots \prec e_n\).
For \(i \in \Bqty{0, \ldots, n - 1}\), the \emph{width of the \(i\)th level} (\emph{\(i\)th width}) of a BDD/ZDD for a subset family of \(E\) with respect to \(\pqty{E, \preceq}\) is defined as the number of nodes labeled with \(e_{i + 1}\).
The \emph{width} of a BDD/ZDD is defined as the maximum width among all levels.
For a matroid \(M\), its \emph{rank function} \(\rankm_M \colon 2^{\fun{\gset}{M}} \to \mathbb{N}\) is defined by \(\rankm_{M}\pqty{X} \coloneqq \max\set{\abs{I}}{\text{\(I \subseteq X\) and \(I \in \isets\)}}\) for \(X \subseteq \fun{\gset}{M}\).
The \emph{connectivity function}~\cite{tutte_1966,SEYMOUR198825} \(\lambda_M \colon 2^{\fun{\gset}{M}} \to \mathbb{N}\) of \(M\) is defined by
\[
  \confun_{M}\pqty{X} \coloneqq \rankm_{M}\pqty{X} + \rankm_{M}\pqty{\fun{\gset}{M} \setminus X} - \rankm\pqty{M}
\]
for \(X \subseteq \fun{\gset}{M}\).

\begin{theorem}\label{thm:intro_width}
  Let \(\gset\) be a set \(\Bqty{e_1, \ldots, e_n}\), \(M\) be a matroid on \(\gset\), and \(\preceq\) be a total order on \(\gset\) such that \(e_1 \prec \cdots \prec e_n\).
  We denote \(\Bqty{e_1, \ldots, e_i}\) by \(\ei\).
  Then the following hold for the \(i\)th width of a BDD/ZDD for \(M\) with respect to \(\pqty{\gset, \preceq}\):
  \begin{itemize}
    \item If \(M\) is free, then the \(i\)th width is at most \(1\) (proved by \cref{cor:width_free}).
    \item If \(M\) is uniform, then the \(i\)th width is at most \(\confun_{M}\pqty{\ei} + 1\) (proved by \cref{thm:width_uniform}).
    \item If \(M\) is a partition or nested matroid, then the \(i\)th width is at most \(2^{\confun_{M}\pqty{\ei}}\) (proved by \cref{thm:width_partition,thm:width_nested}).
  \end{itemize}
\end{theorem}

\begin{figure}[tb]
  \centering
  \begin{tikzpicture}[scale=0.9, every node/.style={transform shape}]
    \def\refscale{0.7}
    \NewDocumentCommand{\classnode}{m m}{%
      \begin{tabular}{@{}c@{}} #1 \\[-3pt] \scalebox{\refscale}{[#2]} \end{tabular}}
    \tikzset{framestyle/.style={draw, ellipse}}
    \node (free_width) {\classnode{\(1\)}{\Cref{cor:width_free}}};
    \node [fit=(free_width), framestyle, inner sep=-0.08cm] (free_frame) {};
    \node [fill=white] at (free_frame.north) (free_text) {Free};
    \node [above=-0.2cm of free_text] (uniform_width) {\classnode{\(\confun\pqty{\ei} + 1\)}{\Cref{thm:width_uniform}}};
    \node [framestyle, fit=(free_frame)(uniform_width), inner sep=-0.1cm] (uniform_frame) {};
    \node [below=-0.03cm, fill=white] at (uniform_frame.north) {Uniform};
    \node [left=1.3cm of uniform_frame, outer sep=-0.55cm] (partition_width) {\classnode{\(2^{\confun\pqty{\ei}}\)}{\Cref{thm:width_partition}}};
    \node [framestyle, fit=(uniform_frame)(partition_width), inner sep=-0.25cm] (partition_frame) {};
    \node [right=-0.2cm, fill=white] at (partition_frame.north west) {Partition};
    \node [right=1.3cm of uniform_frame, outer sep=-0.55cm] (nested_width) {\classnode{\(2^{\confun\pqty{\ei}}\)}{\Cref{thm:width_nested}}};
    \node [framestyle, fit=(uniform_frame)(nested_width), inner sep=-0.25cm] (nested_frame) {};
    \node [left=-0.2cm, fill=white] at (nested_frame.north east) {Nested};
    \node [anchor=south, outer sep=-0.05cm] at ([yshift=2.3cm] uniform_frame.center) (transversal_width) {Unbounded \scalebox{\refscale}{[\Cref{thm:width_transversal}]}};
    \node [framestyle, inner sep=-0.45cm, fit=(partition_frame)(nested_frame)(transversal_width)] (transversal_frame) {};
    \node [fill=white] at (transversal_frame.north) {Transversal};
    \node [anchor=north, outer sep=-0.1cm] at ([yshift=-2.4cm] uniform_frame.center) (laminar_width) {Unbounded \scalebox{\refscale}{[\Cref{thm:width_laminar}]}};
    \node [framestyle, inner sep=-0.45cm, fit=(partition_frame)(nested_frame)(laminar_width)] (laminar_frame) {};
    \node [fill=white] at (laminar_frame.south) {Laminar};
  \end{tikzpicture}
  \caption{Upper bounds on the \(i\)th width depending only on \(\confun\pqty{\ei}\)}
  \label{fig:width_repr_matroid}
\end{figure}

These results, along with others (\cref{thm:width_transversal,thm:width_laminar}), are summarized in \cref{fig:width_repr_matroid}.
The proof of this theorem utilizes the fact that the number of minors of \(M\) on \(\eivar\) bounds the \(i\)th width.
Actually, we show \cref{thm:intro_width} by providing upper bounds on the number of minors on \(\eivar\) for specific classes of matroids.
Classes of matroids for which the number of minors can be bounded with their connectivity functions are called strongly pigeonhole classes~\cite[Definition 2.10]{Funk2022}, and some classes are known to be strongly pigeonhole~\cite{Funk2023}.
Our results reveal new strongly pigeonhole classes.

While the total order \(\preceq\) has been arbitrary so far, certain orders may result in relatively smaller sizes of BDDs/ZDDs.
We show that, for a partition or nested matroid, a certain total order results in an exponentially smaller upper bound on the width than that given in \cref{thm:intro_width}.
Here, for a matroid \(M\), the \emph{pathwidth}~\cite[Section 4]{GEELEN2006405} \(\pathwidth\pqty{M}\) of \(M\) on an \(n\)-element set \(\gset\) is defined by
\[
  \pathwidth\pqty{M} \coloneqq \min\set{\max\set{\confun_{M}\pqty{\ei}}{i \in \Bqty{1, \ldots, n}}}{\text{\(\preceq\) is a total order on \(\gset\)}}\text{.}
\]

\begin{theorem}\label{thm:intro_pathwidth}
  There exists a total order on the ground set such that the width of the BDD/ZDD for a partition or nested matroid \(M\) is at most \(\pathwidth\pqty{M} + 1\) (proved by \cref{thm:pathwidth_partition,thm:pathwidth_nested}).
\end{theorem}

\section{Preliminaries}

Let \(2^E\) denote the power set of a set \(E\), and \(\mathbb{N}\) the set of all non-negative integers.

\subsection{Matroids}

In addition to the definitions of matroids in \cref{sec:introduction}, a subset family \(\bases \subseteq 2^{\gset}\) of the ground set \(\gset\) is the collection of bases of some matroid if and only if it satisfies \cref{item:b1} and the following:
\begin{enumerate}[resume*=base]
  \item\label{item:b2_ast} For \(B_1, B_2 \in \bases\) and \(x \in B_2 \setminus B_1\), there exists \(y \in B_1 \setminus B_2\) such that \(\pqty{B_1 \setminus \Bqty{y}} \cup \Bqty{x} \in \bases\).
\end{enumerate}

For a matroid \(M\), an element \(e \in \fun{\gset}{M}\) is called a \emph{loop} if \(\Bqty{e}\) is not independent and called a \emph{coloop} if every basis of \(M\) contains \(e\).
Recall from \cref{sec:introduction} that \(\rankm_M \colon 2^{\fun{\gset}{M}} \to \mathbb{N}\) denotes the rank function of \(M\).
We often abbreviate \(\rankm_{M}\pqty{\fun{\gset}{M}}\) to \(\rankm\pqty{M}\) and refer to it as the \emph{rank} of \(M\).
For \(X \subseteq E\), we write \(\rankm_{M}\pqty{X}\) as \(\rankm\pqty{X}\) if \(M\) is clear from the context.

Some operations applied to matroids produce different matroids.
Let \(M\) be a matroid on the ground set \(\gset\).
For \(X \subseteq \gset\), let \(\restrict{\isets}{X}\) be the collection of independent sets included in \(X\).
Then the pair \(\pqty{X, \restrict{\isets}{X}}\) forms a matroid, called the \emph{restriction} \(\restrict{M}{X}\) of \(M\) to \(X\) or \emph{deletion} \(\deletion{M}{\pqty{\gset \setminus X}}\) of \(\gset \setminus X\) from \(M\).
Similarly, let \(\isets_X\) be the collection of independent sets \(I\) such that \(I \cup B \in \fun{\isets}{M}\) for some \(B \in \fun{\bases}{\restrict{M}{X}}\).
Then the pair \(\pqty{\gset \setminus X, \isets_X}\) forms a matroid, called the \emph{contraction} \(\contraction{M}{X}\) of \(X\) from \(M\).
It is known that the resulting matroid is uniquely determined regardless of the order of deletions and contractions, that is, \(\contraction{\pqty{\deletion{M}{X}}}{Y} = \deletion{\pqty{\contraction{M}{Y}}}{X}\) for disjoint \(X, Y \subseteq \gset\).
Thus a matroid obtained from \(M\) by a sequence of deletions and contractions can be represented as \(\contraction{\deletion{M}{X}}{Y}\), which is called the \emph{minor} of \(M\) on \(\gset \setminus \pqty{X \cup Y}\).
Let \(M_1, M_2\) be matroids such that \(\fun{\gset}{M_1} \cap \fun{\gset}{M_2} = \emptyset\).
Then \(\pqty{\fun{\gset}{M_1} \cup \fun{\gset}{M_2}, \set{I_1 \cup I_2}{\text{\(I_1 \in \fun{\isets}{M_1}\) and \(I_2 \in \fun{\isets}{M_2}\)}}}\) forms a matroid, called the \emph{direct sum} \(M_1 \oplus M_2\) of \(M_1\) and \(M_2\).

In \cref{sec:introduction}, cycle matroid and \(\mathbb{F}\)-representable matroid are mentioned as classes of matroids.
The following are other classes of matroids appearing in this paper.

\begin{myexample}[Free matroid]\label{eg:free}
  A matroid \(M\) is called \emph{free} if its collection of independent sets is \(2^{\fun{\gset}{M}}\).
\end{myexample}

\begin{myexample}[Uniform matroid and partition matroid]\label{eg:partition}
  A matroid \(M\) is called \emph{uniform} if its collection of independent sets is \(\set{I \subseteq \fun{\gset}{M}}{\abs{I} \leq \rankm\pqty{M}}\).
  A uniform matroid of rank \(r\) on an \(n\)-element set is denoted by \(U_{r, n}\).
  Moreover, a matroid that can be expressed as a direct sum of uniform matroids is called a \emph{partition matroid}.
\end{myexample}

\begin{myexample}[Transversal matroid and nested matroid]\label{eg:nested}
  A matroid \(M\) is called \emph{transversal}~\cite[Section 1]{Edmonds1965} if there exists a sequence \(\pqty{A_1, \ldots, A_m}\) of subsets of \(\fun{\gset}{M}\) such that, for all \(I \subseteq \fun{\gset}{M}\), \(I\) is independent in \(M\) if and only if there is an injection \(\varphi \colon I \to \Bqty{1, \ldots, m}\) satisfying \(e \in A_{\varphi\pqty{e}}\) for every \(e \in I\).
  Such a sequence \(\pqty{A_1, \ldots, A_m}\) is called a \emph{presentation} of \(M\).
  In particular, \(M\) is called \emph{nested}~\citetext{\citealp[Section 8]{Crapo1965}; \citealp[Section 2]{Bonin2008}} if there exists a presentation \(\pqty{A_1, \ldots, A_m}\) such that \(A_1 \subseteq \cdots \subseteq A_m\).
  Note that nested matroids are also called \emph{Schubert matroids}~\cite[Section 3]{Sohoni1999}, \emph{shifted matroids}~\cite[Section 4]{ARDILA200349}, \emph{generalized Catalan matroids}~\cite[Definition 3.7]{BONIN200363}, \emph{freedom matroids}~\cite[Section 5]{CRAPO20051066}, and \emph{PI-matroids}~\cite[Appendix A.1]{BILLERA20091727} (see \cite[Section 4]{BONIN2006701}).
\end{myexample}

\begin{myexample}[Laminar matroid]\label{eg:laminar}
  A subset family \(\mathcal{A} \subseteq 2^E\) of a set \(E\) is called \emph{laminar} if \(A \subseteq B\) or \(B \subseteq A\) for all \(A, B \in \mathcal{A}\) with \(A \cap B \neq \emptyset\).
  A matroid \(M\) is called \emph{laminar} (see \cite{FIFE2017206}) if there exists a laminar family \(\mathcal{A} \subseteq 2^{\fun{\gset}{M}}\) and a function \(c \colon \mathcal{A} \to \mathbb{N}\) such that, for all \(I \subseteq \fun{\gset}{M}\), \(I\) is independent in \(M\) if and only if \(\abs{I \cap A} \leq \fun{c}{A}\) for every \(A \in \mathcal{A}\).
\end{myexample}

\subsection{BDDs and ZDDs}\label{sec:decision_diagrams}

\emph{Binary decision diagrams} (\emph{BDDs})~\cite{10.1109/TC.1986.1676819} and \emph{zero-suppressed binary decision diagrams} (\emph{ZDDs})~\cite{10.1145/157485.164890} are directed acyclic graphs under specific conditions and represent Boolean functions and set families, respectively.
Note that a Boolean function can be viewed as a set family if the set of true variables in an assignment of input variables for which the function evaluates to true belongs to the family.

\begin{figure}[tb]
  \centering
  \begin{tikzpicture}
    \node [terminal] at (2, 0) (bot) {\(\bot\)};
    \node [terminal] at (4, 0) (top) {\(\top\)};
    \foreach \index in {1,2,...,4}
      \node [ddnode] at ({2*(\index-1)}, \ddsep) (node3\index) {\(3\)};
    \draw [zedge] (node31.335) -- (top.north) node [above, pos=0.06] {\edgelabel{0}};
    \draw [oedge] (node31.305) -- (bot.north) node [below, at start] {\edgelabel{1}};
    \draw [zedge] (node32.south west) -- (bot.north) node [left, near start] {\edgelabel{0}};
    \draw [oedge] (node32.south east) -- (bot.north) node [right, near start] {\edgelabel{1}};
    \draw [zedge] (node33) -- (top.north) node [right, pos=0.2] {\edgelabel{0}};
    \draw [oedge] (node33) -- (bot.north) node [left, at start] {\edgelabel{1}};
    \draw [zedge] (node34.205) -- (top.north) node [above, very near start] {\edgelabel{0}};
    \draw [oedge] (node34.235) -- (top.north) node [below, at start] {\edgelabel{1}};
    \foreach \index [
      evaluate=\index as \l using int(2*\index-1),
      evaluate=\index as \r using int(2*\index)
    ] in {1,2} {
      \node [ddnode] at ({1+4*(\index-1)}, 2*\ddsep) (node2\index) {\(2\)};
      \draw [zedge] (node2\index) -- (node3\l) node [left, very near start] {\edgelabel{0}};
      \draw [oedge] (node2\index) -- (node3\r) node [right, very near start] {\edgelabel{1}};
    }
    \node [ddnode] at (3, 3*\ddsep) (node11) {\(1\)};
    \draw [zedge] (node11) -- (node21) node [left, at start] {\edgelabel{0}};
    \draw [oedge] (node11) -- (node22) node [right, at start] {\edgelabel{1}};
  \end{tikzpicture}
  \caption{A binary decision tree for \(\Bqty{\emptyset, \Bqty{1}, \Bqty{1, 2}, \Bqty{1, 2, 3}}\)}
  \label{fig:decision_tree}
\end{figure}
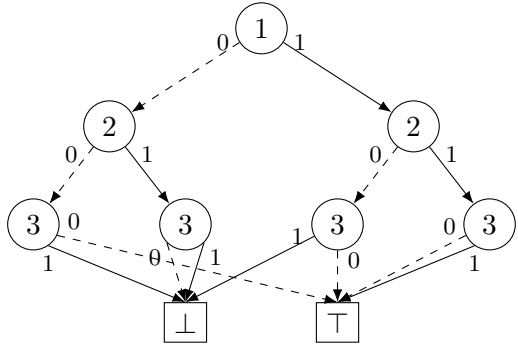

Let \(E\) be a finite set, \(\preceq\) be a total order on \(E\), and \(\mathcal{S} \subseteq 2^E\) be a subset family of \(E\).
The BDD \(\bdd{\mathcal{S}}[E][\preceq]\) and ZDD \(\zdd{\mathcal{S}}[E][\preceq]\) for \(\mathcal{S}\) with respect to \(\pqty{E, \preceq}\) are defined as follows.
Given the family \(\mathcal{S}\) and the totally ordered set \(\pqty{E, \preceq}\), a binary decision tree corresponding to them is uniquely determined.
For example, when \(\mathcal{S}\) is \(\Bqty{\emptyset, \Bqty{1}, \Bqty{1, 2}, \Bqty{1, 2, 3}}\) and \(\pqty{E, \preceq}\) is \(\pqty{\Bqty{1, 2, 3}, \leq}\), \cref{fig:decision_tree} is the corresponding binary decision tree.
The directed acyclic graph obtained by applying the following reduction rules to the binary decision tree is a BDD or ZDD.
The applied rules differentiate BDDs and ZDDs.

As a common condition, the directed acyclic graph has at most two nodes with outdegree zero, called \emph{terminals}.
One of two terminals is called the \emph{0-terminal} \(\bot\) and the other the \emph{1-terminal} \(\top\).
Each non-terminal node \(\nu\) has a label \(\labeldd\pqty{\nu}\) representing an element of \(E\) and exactly two arcs whose tails are \(\nu\).
These arcs are called the \emph{0-arc} and \emph{1-arc}, and their heads are denoted by \(\nu_0\) and \(\nu_1\), respectively.
Since there is a unique 1-arc (resp. 0-arc) whose tail is the same as a given 0-arc (resp. 1-arc), in this paper, it is referred to as the corresponding 1-arc (resp. 0-arc).
The \emph{size} of a BDD/ZDD is the number of non-terminal nodes in it.

\begin{figure}[tb]
  \centering
  \begin{subcaptionblock}{\textwidth}
    \centering
    \begin{tikzpicture}
      \def\newddsep{1.5}
      \node [triangle] at (0, 0) (zddll) {};
      \node [triangle] at (\newddsep, 0) (zddlr) {};
      \node [ddnode] at (0, 1) (nodell) {\(e\)};
      \draw [oedge] (0, 2) -- (nodell);
      \draw [zedge] (nodell) -- (zddll.apex) node [left, near start] {\edgelabel{0}};
      \draw [oedge] (nodell) -- (zddlr.apex) node [right, at start] {\edgelabel{1}};
      \node [ddnode] at (\newddsep, 1) (nodelr) {\(e\)};
      \draw [oedge] (\newddsep, 2) -- (nodelr);
      \draw [zedge] (nodelr) -- (zddll.apex) node [left, at start] {\edgelabel{0}};
      \draw [oedge] (nodelr) -- (zddlr.apex) node [right, near start] {\edgelabel{1}};
      \coordinate (midy) at ($(zddlr.lower side)!0.5!(\newddsep,2)$);
      \draw [reducearrow] ($(midy)+(1,0)$) -- ($(midy)+(2,0)$);
      \node [triangle] at (3+\newddsep, 0) (zddrl) {};
      \node [triangle] at (3+2*\newddsep, 0) (zddrr) {};
      \node [ddnode] at ($(zddrl.apex)!0.5!(zddrr.apex)+(0,1)$) (noder) {\(e\)};
      \draw [oedge] ($(zddrl.apex)+(0,2)$) -- (noder);
      \draw [oedge] ($(zddrr.apex)+(0,2)$) -- (noder);
      \draw [zedge] (noder) -- (zddrl.apex) node [left, very near start] {\edgelabel{0}};
      \draw [oedge] (noder) -- (zddrr.apex) node [right, very near start] {\edgelabel{1}};
    \end{tikzpicture}
    \caption{Node sharing}\label{fig:node_sharing}
    \vspace{0.5\baselineskip}
  \end{subcaptionblock}
  \begin{subcaptionblock}{0.49\textwidth}
    \centering
    \begin{tikzpicture}
      \node [triangle] at (0, 0) (zddl) {};
      \node [ddnode] at (0, 1) (nodel) {\(e\)};
      \draw [oedge] (0, 2) -- (nodel);
      \draw [zedge] (nodel.south west) -- (zddl.apex) node [left, near start] {\edgelabel{0}};
      \draw [oedge] (nodel.south east) -- (zddl.apex) node [right, near start] {\edgelabel{1}};
      \coordinate (midy) at ($(zddl.lower side)!0.5!(0,2)$);
      \draw [reducearrow] ($(midy)+(1,0)$) -- ($(midy)+(2,0)$);
      \node [triangle] at (3, 0) (zddr) {};
      \draw [oedge] (3, 2) -- (zddr.apex);
    \end{tikzpicture}
    \caption{Node deletion for BDDs}\label{fig:node_deletion_bdd}
  \end{subcaptionblock}
  \begin{subcaptionblock}{0.49\textwidth}
    \centering
    \begin{tikzpicture}
      \node [triangle] at (0, 0) (zddl) {};
      \node [anchor=south, terminal] at ($(zddl.lower side)+(1,0)$) (bot) {\(\bot\)};
      \node [ddnode] at (0.5, 1) (nodel) {\(e\)};
      \draw [oedge] (0.5, 2) -- (nodel);
      \draw [zedge] (nodel.south west) -- (zddl.apex) node [left, very near start] {\edgelabel{0}};
      \draw [oedge] (nodel.south east) -- (bot.north) node [right, very near start] {\edgelabel{1}};
      \coordinate (midy) at ($(bot.south)!0.5!(1,2)$);
      \draw [reducearrow] ($(midy)+(1,0)$) -- ($(midy)+(2,0)$);
      \node [triangle] at (4, 0) (zddr) {};
      \draw [oedge] (4, 2) -- (zddr.apex);
    \end{tikzpicture}
    \caption{Node deletion for ZDDs}\label{fig:node_deletion_zdd}
  \end{subcaptionblock}
  \caption{Reduction rules for BDDs and ZDDs}
\end{figure}

The reduction rules are as follows.
\Cref{item:node_sharing} is the common rule as shown in \cref{fig:node_sharing}:
\begin{enumerate}[label=\textbf{(NS)}, ref=(NS)]
  \item\label{item:node_sharing} Non-terminal nodes \(\nu, \nu^\prime\) whose labels are the same and which satisfy both \(\nu_0 = \nu^\prime_0\) and \(\nu_1 = \nu^\prime_1\) must be merged into a single node.
\end{enumerate}
Additionally, BDDs (resp. ZDDs) must satisfy the following rule \cref{item:node_deletion_bdd} (resp. \cref{item:node_deletion_zdd}) as illustrated in \cref{fig:node_deletion_bdd} (resp. \cref{fig:node_deletion_zdd}):
\begin{enumerate}[label=\textbf{(B-ND)}, ref=(B-ND)]
  \item\label{item:node_deletion_bdd} A non-terminal node \(\nu\) with \(\nu_0 = \nu_1\) must be removed, and the arcs whose head is \(\nu\) must be redirected to \(\nu_0\) (equal to \(\nu_1\)).
\end{enumerate}
\vspace{-\baselineskip}
\begin{enumerate}[label=\textbf{(Z-ND)}, ref=(Z-ND)]
  \item\label{item:node_deletion_zdd} A non-terminal node \(\nu\) with \(\nu_1 = \bot\) must be removed, and the arcs whose head is \(\nu\) must be redirected to \(\nu_0\).
\end{enumerate}
Removed nodes are referred to as redundant nodes.
The resulting BDD/ZDD is uniquely determined regardless of the order in which these rules are applied\footnote{Actually, a BDD where the reduction rules are fully applied is originally called a \emph{reduced ordered binary decision diagram} (\emph{ROBDD}), and a directed acyclic graph obtained by applying the reduction rules to the binary decision tree partially is also originally called a BDD. However, it is common to refer to ROBDDs as BDDs, and this paper follows that convention.}, though a BDD and ZDD representing the same set family may not have the same structure.

\begin{figure}[tb]
  \centering
  \begin{subcaptionblock}{0.49\textwidth}
    \centering
    \begin{tikzpicture}
      \node [terminal] at (0, 0) (bot) {\(\bot\)};
      \node [terminal] at (2, 0) (top) {\(\top\)};
      \foreach \id in {2,3,...,5}
        \node [ddnode] at (2, {(6-\id)*\ddsep}) (node\id1) {\(\id\)};
      \node [ddnode] at (0, 3*\ddsep) (node32) {\(3\)};
      \node [ddnode] at (-1, 4*\ddsep) (node22) {\(2\)};
      \node [ddnode] at (1, 5*\ddsep) (node11) {\(1\)};
      \draw [zedge] (node11) -- (node22) node [left, pos=-0.02] {\edgelabel{0}};
      \draw [oedge] (node11) -- (node21) node [right, at start] {\edgelabel{1}};
      \draw [zedge] (node21) -- (bot.north) node [left, pos=0.02] {\edgelabel{0}};
      \draw [oedge] (node21) -- (node31) node [right, near start] {\edgelabel{1}};
      \draw [zedge] (node22) -- (node32) node [right, at start] {\edgelabel{0}};
      \draw [oedge] (node22) -- (bot.north) node [left, pos=0.04] {\edgelabel{1}};
      \draw [zedge] (node31) -- (node41) node [right, near start] {\edgelabel{0}};
      \draw [oedge] (node31) -- (bot.north) node [left, at start] {\edgelabel{1}};
      \draw [zedge] (node32) -- (node51) node [right, at start] {\edgelabel{0}};
      \draw [oedge] (node32) -- (bot.north) node [left, pos=0.04] {\edgelabel{1}};
      \draw [zedge] (node41) -- (node51) node [right, near start] {\edgelabel{0}};
      \draw [oedge] (node41) -- (bot.north) node [left, at start] {\edgelabel{1}};
      \draw [zedge] (node51) -- (bot.north) node [left, pos=-0.04] {\edgelabel{0}};
      \draw [oedge] (node51) -- (top.north) node [right, near start] {\edgelabel{1}};
    \end{tikzpicture}
    \caption{BDD}
  \end{subcaptionblock}
  \begin{subcaptionblock}{0.49\textwidth}
    \centering
    \begin{tikzpicture}
      \node [terminal] at (0, 0) (bot) {\(\bot\)};
      \node [terminal] at (2, 0) (top) {\(\top\)};
      \foreach \x/\id in {2/5,0/4,2/2,1/1}
        \node [ddnode] at (\x, {(6-\id)*\ddsep}) (node\id) {\(\id\)};
      \draw [zedge] (node1) -- (node4) node [left, pos=0.02] {\edgelabel{0}};
      \draw [oedge] (node1) -- (node2) node [right, at start] {\edgelabel{1}};
      \draw [zedge] (node2) -- (bot.north) node [left, pos=0.02] {\edgelabel{0}};
      \draw [oedge] (node2) -- (node5) node [right, pos=0.04] {\edgelabel{1}};
      \draw [zedge] (node4.290) -- (node5) node [left, near start] {\edgelabel{0}};
      \draw [oedge] (node4.340) -- (node5) node [right, at start] {\edgelabel{1}};
      \draw [zedge] (node5) -- (bot.north) node [left, at start] {\edgelabel{0}};
      \draw [oedge] (node5) -- (top.north) node [right, near start] {\edgelabel{1}};
    \end{tikzpicture}
    \caption{ZDD}
  \end{subcaptionblock}
  \caption{Decision diagrams for \(\Bqty{\Bqty{1, 2, 5}, \Bqty{4, 5}, \Bqty{5}}\)}\label{fig:dd_examples}
\end{figure}

A directed path in the BDD/ZDD represents a subset (family) of \(E\).
Now there is only one node with indegree zero, called the \emph{root}.
We denote the root of a BDD \(\bdd{\mathcal{S}}\) (resp. ZDD \(\zdd{\mathcal{S}}\)) as \(\rootdd\pqty{\bdd{\mathcal{S}}}\) (resp. \(\rootdd\pqty{\zdd{\mathcal{S}}}\)).
Moreover, for all non-terminal nodes \(\nu\), both \(\labeldd\pqty{\nu} \prec \labeldd\pqty{\nu_0}\) and \(\labeldd\pqty{\nu} \prec \labeldd\pqty{\nu_1}\) are satisfied, where \(\labeldd\pqty{\bot}\) and \(\labeldd\pqty{\top}\) are special labels that are larger than any element of \(E\).
These conditions ensure that the sequence of elements in any directed path from the root to a terminal is strictly monotonically increasing with respect to \(\preceq\).
In particular, a path from the root to 1-terminal is called a \emph{1-path}, and a set (family) in \(\mathcal{S}\) is represented by a 1-path.
Precisely, for ZDDs, a set corresponding to a 1-path is defined as the label set of the tails of the traversed 1-arcs, and such a set is considered to belong to \(\mathcal{S}\).
The same holds for BDDs, but a single 1-path often represents multiple sets due to \cref{item:node_deletion_bdd}.
A 1-path in the ZDD represents the set
\[
  \set{e \in E}{\text{\(e\) is the tail label of some 1-arc in the 1-path}}\text{.}
\]
On the other hand, that in the BDD represents the family of subsets \(X \subseteq E\) such that
\[
  \set{e \in E}{\text{\(e\) is the tail label of some 1-arc in the 1-path}} \subseteq X
\]
and \(X\) has no member of
\[
  \set{e \in E}{\text{\(e\) is the tail label of some 0-arc in the path}}\text{.}
\]
In other words, the redundancy in BDDs means don't-care.
For example, \cref{fig:dd_examples} shows a BDD and ZDD representing \(\Bqty{\Bqty{1, 2, 5}, \Bqty{4, 5}, \Bqty{5}}\) with respect to \(\pqty{\Bqty{1, 2, 3, 4, 5}, \leq}\).
In this paper, a directed path is treated as a sequence of labeled arcs.

The subgraph rooted by a node in a BDD \(\bdd{\mathcal{S}}\) (resp. ZDD \(\zdd{\mathcal{S}}\)) is also a BDD (resp. ZDD), called the \emph{sub-BDD} of \(\bdd{\mathcal{S}}\) (resp. \emph{sub-ZDD} of \(\zdd{\mathcal{S}}\)).
In particular, if the root of \(\bdd{\mathcal{S}}\) is non-terminal, the sub-BDD rooted by \(\rootdd\pqty{\bdd{\mathcal{S}}}_0\) (resp. \(\rootdd\pqty{\bdd{\mathcal{S}}}_1\)) is called the \emph{0-child} \(\childz\pqty{\bdd{\mathcal{S}}}\) (resp. \emph{1-child} \(\childo\pqty{\bdd{\mathcal{S}}}\)) of \(\bdd{\mathcal{S}}\).
The same definitions apply to ZDDs for 0-child and 1-child.
For a ZDD \(\zdd{\mathcal{S}}\) with a non-terminal root, given the conditions described above, the 0-child and 1-child of \(\zdd{\mathcal{S}}\) represent \(\set{S \in \mathcal{S}}{\labelroot{\zdd{\mathcal{S}}} \notin S}\), \(\set{S \setminus \Bqty{\labelroot{\zdd{\mathcal{S}}}}}{\labelroot{\zdd{\mathcal{S}}} \in S \in \mathcal{S}}\), respectively.

The following property is known for BDDs/ZDDs:

\begin{proposition}\label{prop:reachablity_dd}
  Every non-terminal node in a BDD/ZDD can reach 1-terminal.
\end{proposition}

\begin{proof}
  Suppose that there exist non-terminal nodes that cannot reach 1-terminal.
  Let \(\nu\) be one of such nodes with the largest label.
  Since \(\labeldd\pqty{\nu} \prec \labeldd\pqty{\nu_0}\) and \(\labeldd\pqty{\nu} \prec \labeldd\pqty{\nu_1}\), we have \(\nu_0 = \nu_1 = \bot\).
  This contradicts \cref{item:node_deletion_bdd} for BDDs and \cref{item:node_deletion_zdd} for ZDDs.
\end{proof}

\section{Structural relations between different representations}\label{sec:structural_relations}

There are several variations of BDDs and ZDDs representing matroids.
In addition to the choice between BDDs and ZDDs, the represented collection of independent sets or bases must be chosen.
Moreover, there may be some relations between the BDD/ZDD for a matroid and its dual.
This section compares these three choices and proves \cref{thm:intro_structure}.

Fix a matroid \(M\) and a total order \(\preceq\) on the ground set of \(M\).
In this section, we denote by \(\gset\), \(\isets\), and \(\bases\) the ground set, the collections of independent sets, and that of bases of \(M\), respectively.

Before moving to the proof, we introduce clutters.
A set family is called a \emph{clutter} if no member of it contains another member.
We can show that each 1-path corresponds to a member of a clutter one-to-one in a BDD for the clutter.

\begin{lemma}\label{lem:clutter_bdd}
  Let \(\mathcal{C} \subseteq 2^E\) be a clutter on \(E\).
  In the BDD \(\bdd{\mathcal{C}}[E][\preceq]\), every 1-path traverses exactly \(\abs{E}\) non-terminal nodes.
\end{lemma}

\begin{proof}
  Suppose that there exists a 1-path with redundant nodes.
  Let \(e \in E\) be the label of one of such nodes.
  Then, for some \(C \subseteq E \setminus \Bqty{e}\), both \(C\) and \(C \cup \Bqty{e}\) are members of \(\mathcal{C}\); a contraction.
\end{proof}

Similarly, we can show that the heads of the corresponding 0-arc and 1-arc are different in a ZDD for a clutter.

\begin{lemma}\label{lem:clutter_zdd}
  Let \(\mathcal{C} \subseteq 2^E\) be a clutter on \(E\).
  In the ZDD \(\zdd{\mathcal{C}}[E][\preceq]\), every non-terminal node \(\nu\) satisfies that \(\nu_0 \neq \nu_1\).
\end{lemma}

\begin{proof}
  Suppose that there exists a non-terminal node \(\nu\) such that \(\nu_0 = \nu_1\).
  Since \cref{item:node_deletion_zdd} implies that \(\nu_1 \neq \bot\), there is a 1-path in the ZDD traversing the 1-arc \(\pqty{\nu, \nu_1}\) by \cref{prop:reachablity_dd}.
  Let \(C \in \mathcal{C}\) be the corresponding set to the 1-path.
  Since \(\nu_0 = \nu_1\), replacing \(\pqty{\nu, \nu_1}\) in the 1-path with the 0-arc \(\pqty{\nu, \nu_0}\) also results in a 1-path in the ZDD.
  Thus \(C \setminus \Bqty{\labeldd\pqty{\nu}} \in \mathcal{C}\); a contraction.
\end{proof}

A collection of bases is a clutter since each basis has the same size~\cite[Theorem 6]{b2d08917-c364-3337-af1c-57f8ed647c1c}.
Thus \cref{lem:clutter_bdd,lem:clutter_zdd} can be applied to a collection of bases.

The first comparison focuses on the sizes between a BDD and a ZDD for a collection of bases.
This comparison can be generalized to a clutter.

\begin{theorem}\label{thm:clutter_zdd_and_clutter_bdd}
  Let \(\mathcal{C} \subseteq 2^E\) be a clutter on \(E\).
  The size of the ZDD \(\zdd{\mathcal{C}}[E][\preceq]\) is never greater than that of the BDD \(\bdd{\mathcal{C}}[E][\preceq]\).
\end{theorem}

\begin{proof}
  By \cref{lem:clutter_bdd}, the BDD \(\bdd{\mathcal{C}}[E][\preceq]\) can be interpreted as a ZDD where \cref{item:node_deletion_zdd} is not fully applied.
\end{proof}

The second comparison focuses on the relation between a matroid and its dual in a BDD for a collection of bases.
This comparison can also be generalized to a clutter.

\begin{theorem}\label{thm:clutter_bdd}
  Let \(\mathcal{C} \subseteq 2^E\) be a clutter on \(E\) and \(\mathcal{C}^\ast\) be the subset family \(\set{E \setminus C}{C \in \mathcal{C}}\) of \(E\).
  The BDD \(\bdd{\mathcal{C}^\ast}[E][\preceq]\) can be obtained from \(\bdd{\mathcal{C}}[E][\preceq]\) by updating \(\pqty{\nu_0, \nu_1} \gets \pqty{\nu_1, \nu_0}\), that is, swapping the 0- and 1-arc for every non-terminal node \(\nu\).
\end{theorem}

\begin{proof}
  The structure of the BDD remains unchanged except for the labels after the update.
  By \cref{lem:clutter_bdd}, let \(C \in \mathcal{C}\) be the corresponding set to some 1-path in \(\bdd{\mathcal{C}}[E][\preceq]\).
  After the update, this 1-path changes to represent only \(E \setminus C\).
  Therefore, the updated BDD represents \(\set{E \setminus C}{C \in \mathcal{C}}\), which is precisely \(\mathcal{C}^\ast\).
\end{proof}

The third comparison focuses on the relation between decision diagrams for a collection of independent sets and those for a collection of bases.
Here the lexicographic order on \(\bases\) with respect to the total order \(\preceq\) is defined as follows.
For \(B \in \bases\), let \(\sequence\pqty{B}\) be the sequence of elements in \(B\) arranged in ascending order with respect to \(\preceq\) and \({\sequence\pqty{B}}_j\) be the \(j\)th element of \(\sequence\pqty{B}\).
For \(B_1, B_2 \in \bases\), we say that \(B_1\) is lexicographically smaller than \(B_2\) if there exists \(k \in \Bqty{1, \ldots, \rankm\pqty{M}}\) such that \({\sequence\pqty{B_1}}_k \prec {\sequence\pqty{B_2}}_k\) and \({\sequence\pqty{B_1}}_j = {\sequence\pqty{B_2}}_j\) for all \(j \in \Bqty{1, \ldots, k - 1}\).

\begin{theorem}\label{thm:isets_and_bases}
  Let \(\isets^\prime\) be a set family represented by the BDD obtained from \(\bdd{\bases}[\gset][\preceq]\) by updating \(\nu_0 \gets \nu_1\) for every non-terminal node \(\nu\) with \(\nu_0 = \bot\) and then applying the reduction rules for BDDs.
  In other words, \(\bdd{\isets^\prime}\) is obtained from \(\bdd{\bases}[\gset][\preceq]\) by replacing every 0-arc whose head is 0-terminal with the corresponding 1-arc and applying the reduction rules for BDDs.
  Then \(\isets^\prime\) is equivalent to \(\isets\).
  The same holds for ZDDs, but the ZDD satisfies the reduction rules immediately after the update.
\end{theorem}

\begin{proof}
  Let \(\mathsf{B}^\prime\) be the structure obtained from \(\bdd{\bases}[\gset][\preceq]\) by updating \(\nu_0 \gets \nu_1\) for every non-terminal node \(\nu\) with \(\nu_0 = \bot\) and not applying the reduction rules.
  Note that \(\mathsf{B}^\prime\) and \(\bdd{\isets^\prime}\) represent the same set family \(\isets^\prime\).

  Take any \(I \in \isets^\prime\).
  Then, by \cref{lem:clutter_bdd}, there is a 1-path \(P\) in \(\mathsf{B}^\prime\) corresponding to \(I\).
  Since the updated 0-arcs have the same head as their corresponding 1-arcs, there is a 1-path \(P^\prime\) in \(\mathsf{B}^\prime\) obtained by replacing each updated 0-arc \(\pqty{\nu, \nu_0}\) in \(P\) with the corresponding 1-arc \(\pqty{\nu, \nu_1}\).
  The corresponding set \(B \in \isets^\prime\) to \(P^\prime\) satisfies that \(I \subseteq B\).
  Moreover, this 1-path \(P^\prime\) also exists in \(\bdd{\bases}[\gset][\preceq]\), implying \(B \in \bases\).
  Therefore, we have \(I \in \isets\).
  This implies that \(\isets^\prime \subseteq \isets\).

  Conversely, take any \(I \in \isets\).
  Let \(B\) be the lexicographically largest basis of \(M\) with respect to \(\preceq\) such that \(I \subseteq B\) and \(P\) be the 1-path in \(\bdd{\bases}[\gset][\preceq]\) corresponding to \(B\).

  Take any \(e \in B \setminus I\).
  By \cref{lem:clutter_bdd}, there exists exactly one non-terminal node \(\nu_e\) in \(P\) such that \(e = \labeldd\pqty{\nu_e}\).
  Suppose to the contrary that \(\pqty{\nu_e}_0 \neq \bot\).
  By \cref{prop:reachablity_dd}, there are 1-paths in \(\bdd{\bases}[\gset][\preceq]\) identical to \(P\) up to \(\nu_e\) but traversing the 0-arc \(\pqty{\nu_e, \pqty{\nu_e}_0}\).
  Let \(B_\succ \in \bases\) be a basis of \(M\) corresponding to such a 1-path so that, among all such bases, \(\abs{B_\succ \cap I}\) is maximal.
  Since \(B_\succ\) is lexicographically larger than \(B\), the choice of \(B\) implies \(I \nsubseteq B_\succ\).
  By \cref{item:b2_ast}, for some \(x \in I \setminus B_\succ \subseteq B \setminus B_\succ\), there exists \(y \in B_\succ \setminus B\) such that \(\pqty{B_\succ \setminus \Bqty{y}} \cup \Bqty{x} \in \bases\).
  Since \(y \notin I \subseteq B\), we have \(\abs{\pqty{\pqty{B_\succ \setminus \Bqty{y}} \cup \Bqty{x}} \cap I} = \abs{B_\succ \cap I} + 1\).
  Moreover, \(x \neq e\) since \(e \notin I\).
  Hence the maximality of \(\abs{B_\succ \cap I}\) is contradicted.
  Thus we obtain that \(\pqty{\nu_e}_0 = \bot\).

  Since \(\pqty{\nu_e}_0 = \bot\) for each \(e \in B \setminus I\), the 1-arcs in \(P\) whose tails have labels in \(B \setminus I\) satisfy that all corresponding 0-arcs are updated.
  Thus there is a 1-path in \(\mathsf{B}^\prime\) representing \(I\), implying that \(I \in \isets^\prime\).
  Therefore, we have \(\isets \subseteq \isets^\prime\).

  Replacing ``by \cref{lem:clutter_bdd}'' with ``by definition'', the same result can be obtained for ZDDs.
  Finally, we show that the updated ZDD \(\zdd{\isets^\prime}\) satisfies the reduction rules.
  Since the update only changes 0-arcs whose head is 0-terminal, all non-terminal nodes in \(\zdd{\isets^\prime}\) are reachable from the root, and no 1-arcs ends in 0-terminal.
  Suppose that there exist different non-terminal nodes \(\nu, \nu^\prime\) in \(\zdd{\bases}[\gset][\preceq]\) that must be shared after the update.
  Since all 1-arcs are not updated, we have \(\nu_1 = \nu^\prime_1\).
  Thus \(\nu_0 \neq \nu^\prime_0\), and either \(\nu_0 = \nu_1\) or \(\nu^\prime_0 = \nu^\prime_1\) must hold, which contradicts \cref{lem:clutter_zdd}.
\end{proof}

\begin{figure}[tb]
  \centering
  \begin{subcaptionblock}{0.41\textwidth}
    \centering
    \begin{tikzpicture}
      \node [terminal] at (-1, 0) (bot) {\(\bot\)};
      \node [terminal] at (1, 0) (top) {\(\top\)};
      \foreach \x/\id in {0/4,-1/3,1/2,0/1}
        \node [ddnode] at (\x, {(5-\id)*\ddsep}) (node\id) {\(\id\)};
      \draw [zedge] (node1) -- (node3) node [left, pos=0.02] {\edgelabel{0}};
      \draw [oedge] (node1) -- (node2) node [right, at start] {\edgelabel{1}};
      \draw [zedge] (node2) -- (node4) node [left, pos=0.02] {\edgelabel{0}};
      \draw [oedge] (node2) -- (top.north) node [right, pos=0.04] {\edgelabel{1}};
      \draw [zedge] (node3) -- (bot.north) node [left, pos=0.08] {\edgelabel{0}};
      \draw [oedge] (node3.340) -- (node4) node [right, at start] {\edgelabel{1}};
      \draw [zedge] (node4.200) -- (bot.north) node [left, at start] {\edgelabel{0}};
      \draw [oedge] (node4.340) -- (top.north) node [right, at start] {\edgelabel{1}};
    \end{tikzpicture}
    \caption{A ZDD for \(\Bqty{\Bqty{1, 2}, \Bqty{1, 4}, \Bqty{3, 4}}\)}
  \end{subcaptionblock}
  \begin{subcaptionblock}{0.58\textwidth}
    \centering
    \begin{tikzpicture}
      \node [terminal] at (1, 0) (top) {\(\top\)};
      \foreach \x/\id in {0/4,-1/3,1/2,0/1}
        \node [ddnode] at (\x, {(5-\id)*\ddsep}) (node\id) {\(\id\)};
      \draw [zedge] (node1) -- (node3) node [left, pos=0.02] {\edgelabel{0}};
      \draw [oedge] (node1) -- (node2) node [right, at start] {\edgelabel{1}};
      \draw [zedge] (node2) -- (node4) node [left, pos=0.02] {\edgelabel{0}};
      \draw [oedge] (node2) -- (top.north) node [right, pos=0.04] {\edgelabel{1}};
      \draw [zedge] (node3.290) -- (node4) node [left, near start] {\edgelabel{0}};
      \draw [oedge] (node3.340) -- (node4) node [right, at start] {\edgelabel{1}};
      \draw [zedge] (node4.290) -- (top.north) node [left, near start] {\edgelabel{0}};
      \draw [oedge] (node4.340) -- (top.north) node [right, at start] {\edgelabel{1}};
    \end{tikzpicture}
    \caption{A ZDD for \(\Bqty{\emptyset, \Bqty{1}, \Bqty{3}, \Bqty{4}, \Bqty{1, 2}, \Bqty{1, 4}, \Bqty{3, 4}}\)}
  \end{subcaptionblock}
  \caption{A counterexample to \cref{thm:isets_and_bases} in clutters}\label{fig:counterexample_of_isets_and_bases}
\end{figure}

The last sentence of this theorem indicates that the size of the ZDD does not change before and after the update.
In other words, the sizes of \(\zdd{\bases}[\gset][\preceq]\) and \(\zdd{\isets}[\gset][\preceq]\) are equal.
Note that applying this operation to the BDD/ZDD for a clutter \(\mathcal{C}\) does not necessarily yield the BDD/ZDD for \(\set{X}{X \subseteq C \in \mathcal{C}}\).
For example, \cref{fig:counterexample_of_isets_and_bases} shows that applying this operation to the ZDD for \(\Bqty{\Bqty{1, 2}, \Bqty{1, 4}, \Bqty{3, 4}}\) yields the ZDD for \(\Bqty{\emptyset, \Bqty{1}, \Bqty{3}, \Bqty{4}, \Bqty{1, 2}, \Bqty{1, 4}, \Bqty{3, 4}}\), which does not contain \(\Bqty{2}\).
Therefore, a BDD (resp. ZDD) for the collection of independent sets can be obtained from a BDD (resp. ZDD) for the collection of bases in time proportional to the size of the latter, without general operations to transform it into the former.

Finally, we focus on the relation between a ZDD for the collection of bases and a BDD for the collection of independent sets of its dual.
As in \cref{thm:clutter_bdd}, the operation of swapping the 0- and 1-arc is also important in this theorem.
Let \(\isets^\ast\) denote the collection of independent sets of \(M^\ast\).

\begin{theorem}\label{thm:bases_zdd_and_dual_isets_bdd}
  The BDD \(\bdd{\isets^\ast}[\gset][\preceq]\) can be obtained from the ZDD \(\zdd{\bases}[\gset][\preceq]\) by updating \(\pqty{\nu_0, \nu_1} \gets \pqty{\nu_1, \nu_0}\) for every non-terminal node \(\nu\).
\end{theorem}

\begin{proof}
  By \Cref{lem:clutter_zdd}, there are no redundant nodes immediately after the update.
  Hence the updated structure can be interpreted as a BDD.
  Let \(\bdd{\isets^\prime}\) be the updated BDD.

  Let \(P\) be a 1-path in \(\zdd{\bases}[\gset][\preceq]\) and \(B \in \bases\) be the corresponding basis to \(P\).
  Assume that \(J \subseteq \gset \setminus B\) is the set of labels that do not appear in \(P\).
  Then the 1-path in \(\bdd{\isets^\prime}\) corresponding to the updated \(P\) represents the family of sets \(I^\ast\) satisfying \(\gset \setminus \pqty{B \cup J} \subseteq I^\ast \subseteq \gset \setminus B\).
  Note that \(\gset \setminus \pqty{B \cup J}\) means the label set of tails of all 0-arcs traversed by \(P\).
  Since such \(I^\ast\) belongs to \(\isets^\ast\), we have \(\isets^\prime \subseteq \isets^\ast\).

  Conversely, take any \(I^\ast \in \isets^\ast\).
  Let \(B\) be the lexicographically smallest basis of \(M\) with respect to \(\preceq\) such that \(I^\ast \subseteq \gset \setminus B\) and \(P\) be the 1-path in \(\zdd{\bases}[\gset][\preceq]\) corresponding to \(B\).

  Suppose to the contrary that there exists \(e \in \gset \setminus \pqty{B \cup I^\ast}\) with a non-terminal node \(\nu_e\) in \(P\) with \(e = \labeldd\pqty{\nu_e}\).
  Since \(\pqty{\nu_e}_1 \neq \bot\) by \cref{item:node_deletion_zdd}, there are 1-paths in \(\zdd{\bases}[\gset][\preceq]\) identical to \(P\) up to \(\nu_e\) but traversing the 1-arc \(\pqty{\nu_e, \pqty{\nu_e}_1}\) by \cref{prop:reachablity_dd}.
  Let \(B_\prec \in \bases\) be a basis of \(M\) corresponding to such a 1-path so that, among all such bases, \(\abs{B_\prec \cap I^\ast}\) is minimal.
  Since \(B_\prec\) is lexicographically smaller than \(B\), the choice of \(B\) implies \(I^\ast \nsubseteq \gset \setminus B_\prec\), that is, \(B_\prec \cap I^\ast \neq \emptyset\).
  By \cref{item:b2}, for some \(x \in B_\prec \cap I^\ast \subseteq B_\prec \setminus B\), there exists \(y \in B \setminus B_\prec\) such that \(\pqty{B_\prec \setminus \Bqty{x}} \cup \Bqty{y} \in \bases\).
  Since \(y \notin I^\ast \subseteq \gset \setminus B\), we have \(\abs{\pqty{\pqty{B_\prec \setminus \Bqty{x}} \cup \Bqty{y}} \cap I^\ast} = \abs{B_\prec \cap I^\ast} - 1\).
  Moreover, \(x \neq e\) since \(e \notin I^\ast\).
  Hence the minimality of \(\abs{B_\prec \cap I^\ast}\) is contradicted.

  By the above argument, \(P\) traverses no non-terminal nodes with label \(e\) in \(\zdd{\bases}[\gset][\preceq]\) for all \(e \in \gset \setminus \pqty{B \cup I^\ast}\).
  Similar to earlier observations, the 1-path in \(\bdd{\isets^\prime}\) corresponding to the updated \(P\) represents the set family that contains \(I^\ast\).
  Therefore, there is a 1-path in \(\bdd{\isets^\prime}\) representing \(I^\ast\), and we have \(\isets^\ast \subseteq \isets^\prime\).
\end{proof}

\Cref{thm:clutter_zdd_and_clutter_bdd,thm:clutter_bdd,thm:isets_and_bases,thm:bases_zdd_and_dual_isets_bdd} collectively establish \cref{thm:intro_structure}.
\Cref{fig:dd_repr_matroid} focuses on the relations of the sizes among the eight different representations.

\section{Widths of BDDs and ZDDs}\label{sec:width}

Let \(\gset\) be a set \(\Bqty{e_1, \ldots, e_n}\), \(M\) be a matroid on \(\gset\), and \(\preceq\) be a total order on \(\gset\) such that \(e_1 \prec \cdots \prec e_n\).
Recall from \cref{sec:introduction} that, for \(i \in \Bqty{0, \ldots, n - 1}\), the \(i\)th width of a BDD/ZDD for a subset family of \(\gset\) with respect to \(\pqty{\gset, \preceq}\) is defined as the number of non-terminal nodes labeled with \(e_{i + 1}\), and the width of the BDD/ZDD is defined as the maximum width among all levels.
Moreover, we denote \(\Bqty{e_1, \ldots, e_i}\) by \(\ei\).
This section analyzes the widths of BDDs/ZDDs representing matroids.

Before the analysis, we focus on the structure of the ZDD \(\zdd{\fun{\isets}{M}}[\gset][\preceq]\) for the collection \(\fun{\isets}{M}\) of independent sets.
If the label of the root is \(e_1\), then the 0-child of \(\zdd{\fun{\isets}{M}}[\gset][\preceq]\) represents
\[
  \set{I \in \fun{\isets}{M}}{e_1 \notin I}
  = \set{I \subseteq \gset \setminus \Bqty{e_1}}{I \in \fun{\isets}{M}}
  = \fun{\isets}{\deletion{M}{\Bqty{e_1}}}\text{.}
\]
Similarly, the 1-child represents
\[
  \set{I \setminus \Bqty{e_1}}{e_1 \in I \in \fun{\isets}{M}}
  = \set{I \subseteq \gset \setminus \Bqty{e_1}}{I \cup \Bqty{e_1} \in \fun{\isets}{M}}
  = \fun{\isets}{\contraction{M}{\Bqty{e_1}}}
\]
since \(\Bqty{e_1}\) is a basis of \(\restrict{M}{\Bqty{e_1}}\) by \Cref{prop:reachablity_dd}.
Thus, for \(i \in \Bqty{0, \ldots, n - 1}\), the sub-ZDD of \(\zdd{\fun{\isets}{M}}[\gset][\preceq]\) rooted by some non-terminal node with label \(e_{i + 1}\) represents the collection of independent sets of some minor of \(M\) on \(\eivar\).
Moreover, the sub-ZDDs rooted by different non-terminal nodes with label \(e_{i + 1}\) represent different collections of independent sets on \(\eivar\).
Therefore, the \(i\)th width of \(\zdd{\fun{\isets}{M}}[\gset][\preceq]\) is bounded by the number of minors of \(M\) on \(\eivar\).
Similarly, for \(\zdd{\fun{\bases}{M}}[\gset][\preceq]\), \(\bdd{\fun{\isets}{M}}[\gset][\preceq]\), and \(\bdd{\fun{\bases}{M}}[\gset][\preceq]\), the \(i\)th width is also bounded by the number of minors on \(\eivar\).
Note that, as mentioned in \cref{sec:related_work}, this interpretation generalizes the discussion by \citet{10.1007/BFb0015427} to general matroids.

The width can vary depending on the total order \(\preceq\) on \(\gset\).
For several classes of matroids, \cref{sec:width_with_lambda} provides upper bounds on widths with respect to any total order, and \cref{sec:width_with_pathwidth} provides for orders that make the width small.
Moreover, \cref{sec:rank_oracle} shows how to implement a rank oracle using the ZDD structure described above.

\subsection{Upper bounds on widths depending only on connectivity functions}\label{sec:width_with_lambda}

A trivial upper bound for the number of minors on \(\eivar\) is \(2^i\).
In this section, we provide more precise upper bounds in terms of connectivity functions of matroids by proving \cref{thm:intro_width}.
Recall from \cref{sec:introduction} that the connectivity function \(\lambda_M \colon 2^{\gset} \to \mathbb{N}\) of a matroid \(M\) is defined by \(\confun_{M}\pqty{X} \coloneqq \rankm_{M}\pqty{X} + \rankm_{M}\pqty{\gset \setminus X} - \rankm\pqty{M}\) for \(X \subseteq \gset\).
We write \(\lambda_M\) as \(\lambda\) if \(M\) is clear from the context.
This section finds upper bounds on the number of minors that depends only on \(\confun\pqty{\ei}\).

\subsubsection{Related work}\label{sec:related_work}

\Citet{10.1007/BFb0015427} provided an upper bound for cycle matroids.
In the BDD for the family of spanning trees for a simple connected undirected graph \(G\) with the edge set \(E\), each node with label \(e_i\) corresponds to a certain minor of \(G\) on \(\eivar\).
Moreover, limited \(2\)-isomorphic minors correspond to the same node.
Thus the \(i\)th width of the BDD for all spanning trees of \(G\) is equal to the number of minors of \(G\) on \(\eivar\).
This result can be interpreted regarding to BDDs for the collection of bases of the cycle matroid \(\fun{M}{G}\).
Eventually, the number of minors of \(\fun{M}{G}\) on \(\eivar\) is bounded by the \(\pqty{\fun{V}{G \bqty{\ei}} \cap \fun{V}{G \bqty{\eivar}}}\)-th Bell number, where \(G \bqty{X}\) denotes the induced subgraph of \(G\) on \(X \subseteq E\).

The class of matroids known as the strongly pigeonhole class is closely related to this study.
For a matroid \(M\) and \(S \subseteq \fun{\gset}{M}\), define the equivalence relation \(\sim_{M, S}\) on \(2^S\) as follows.
For \(X_1, X_2 \subseteq S\), the sets \(X_1\) and \(X_2\) are equivalent, written \(X_1 \sim_{M, S} X_2\), if and only if the independence of \(X_1 \cup Y\) and \(X_2 \cup Y\) are equivalent in \(M\) for all \(Y \subseteq \fun{\gset}{M} \setminus S\).
A class \(\mathcal{M}\) of matroids is called \emph{strongly pigeonhole}~\cite[Definition 2.10]{Funk2022} if there is a function \(f \colon \mathbb{N} \setminus \Bqty{0} \to \mathbb{N} \setminus \Bqty{0}\) such that, for any matroid \(M \in \mathcal{M}\), positive integer \(\lambda^\prime\), and \(S \subseteq \fun{\gset}{M}\) with \(\confun_{M}\pqty{S} \leq \lambda^\prime\), the size of the quotient set by \(\sim_{M, S}\) is always at most \(\fun{f}{\lambda^\prime}\).
Thus, if a class \(\mathcal{M}\) of matroids is strongly pigeonhole, then the width of BDDs/ZDDs for a matroid \(M \in \mathcal{M}\) can be bounded by a value depending only on the connectivity function of \(M\).
\Citet[Proposition 10 and Theorems 18, 22, and 34]{Funk2023} showed that several classes of matroids are strongly pigeonhole such as fundamental transversal matroids, \(3\)-connected bicircular matroids, and \(3\)-connected \(H\)-gain-graphic matroids for a finite group \(H\).

As a folklore, for a prime power \(q\), the number of minors of a \(\finitefield\pqty{q}\)-representable matroid \(M\) on \(\eivar\) is at most \(\sum_{d = 0}^{\confun_{M}\pqty{\ei}} \qbinom{\confun_{M}\pqty{\ei}}{d}_q\) (see \cite[Theorem 2]{KRAL2012913} and \cite[Theorems 18]{Funk2023}).
Here the Gaussian binomial coefficient \(\qbinom{n}{k}_q\) is defined by
\[
  \qbinom{n}{k}_q \coloneqq \frac{\pqty{q^n - 1} \pqty{q^{n - 1} - 1} \cdots \pqty{q^{n - k + 1} - 1}}{\pqty{q - 1} \pqty{q^2 - 1} \cdots \pqty{q^k - 1}}
\]
for \(n, k \in \mathbb{N}\) with \(0 \leq k \leq n\).

\subsubsection{Proof of \texorpdfstring{\cref{thm:intro_width}}{Theorem \ref{thm:intro_width}}}

The following holds for uniform matroids (see \cref{eg:partition}):

\begin{theorem}\label{thm:width_uniform}
  Let \(M\) be a uniform matroid on the ground set \(\gset\).
  The number of minors of \(M\) on \(\eivar\) is \(\confun_{M}\pqty{\ei} + 1\).
\end{theorem}

\begin{proof}
  For \(X \subseteq \ei\), the collection of bases of \(\restrict{M}{X}\) is the family of subsets \(B \subseteq X\) of size \(\min\Bqty{\abs{X}, \rankm\pqty{M}}\).
  Hence the collection of independent sets of \(\restrict{M}{X}\) is the family of subsets \(I \subseteq \eivar\) such that \(\abs{I} + \min\Bqty{\abs{X}, \rankm\pqty{M}} \leq \rankm\pqty{M}\), that is, \(\abs{I} \leq \max\Bqty{\rankm\pqty{M} - \abs{X}, 0}\).
  The collection of independent sets on \(\eivar\) depends on the size of \(X\):
  \begin{itemize}
    \item If \(\abs{\eivar} \leq \rankm\pqty{M} - \abs{X}\), then the collection is \(2^{\eivar}\).
    \item If \(\rankm\pqty{M} - \abs{\eivar} < \abs{X} < \rankm\pqty{M}\), then the collection is \(\set{I \subseteq \eivar}{\abs{I} \leq \rankm\pqty{M} - \abs{X}}\).
          Note that \(\eivar\) is not independent in \(\minori{X}\) since \(\abs{X} > \rankm\pqty{M} - \abs{\eivar}\), and all subsets of \(\eivar\) of size one are independent in \(\minori{X}\) since \(\abs{X} < \rankm\pqty{M}\).
    \item If \(\abs{X} \geq \rankm\pqty{M}\), then the only independent set in \(\minori{X}\) is \(\emptyset\).
  \end{itemize}
  Therefore, the number of minors on \(\eivar\) is
  \[
    \begin{split}
      &\min\Bqty{\abs{\ei}, \rankm\pqty{M}} - \max\Bqty{\rankm\pqty{M} - \abs{\eivar}, 0} + 1 \\
      &= \min\Bqty{\abs{\ei}, \rankm\pqty{M}} + \min\Bqty{\abs{\eivar}, \rankm\pqty{M}} - \rankm\pqty{M} + 1 \\
      &= \rankm_{M}\pqty{\ei} + \rankm_{M}\pqty{\eivar} - \rankm\pqty{M} + 1 = \confun_{M}\pqty{\ei} + 1\text{.}\qedhere
    \end{split}
  \]
\end{proof}

\begin{myremark}
  This theorem can be derived from \citet[Proposition 10]{Funk2023}.
  However, the proposition focuses on strongly pigeonhole classes and does not directly imply the number of minors.
  \Cref{thm:width_uniform} explicitly evaluates the number of minors of uniform matroids.
\end{myremark}

For a free matroid (see \cref{eg:free}) \(M\) on the ground set \(\gset\) and all \(X \subseteq E\), the value \(\confun_{M}\pqty{X}\) is zero.
Thus we obtain the following since all free matroids are uniform:

\begin{corollary}\label{cor:width_free}
  Let \(M\) be a free matroid on the ground set \(\gset\).
  The number of minors of \(M\) on \(\eivar\) is one.
\end{corollary}

Now the following holds regarding the direct sum of matroids:

\begin{lemma}\label{lem:direct_sum_width}
  Let \(M_1\) and \(M_2\) be matroids with \(\fun{\gset}{M_1} \cap \fun{\gset}{M_2} = \emptyset\).
  For \(X_1 \subseteq \fun{\gset}{M_1}\) (resp. \(X_2 \subseteq \fun{\gset}{M_2}\)), let \(m_1\) (resp. \(m_2\)) be the number of minors of \(M_1\) (resp. \(M_2\)) on \(\fun{\gset}{M_1} \setminus X_1\) (resp. \(\fun{\gset}{M_2} \setminus X_2\)).
  Then the number of minors of \(M_1 \oplus M_2\) on \(\pqty{\fun{\gset}{M_1} \cup \fun{\gset}{M_2}} \setminus \pqty{X_1 \cup X_2}\) is \(m_1 m_2\).
\end{lemma}

\begin{proof}
  The number of minors on \(\pqty{\fun{\gset}{M_1} \cup \fun{\gset}{M_2}} \setminus \pqty{X_1 \cup X_2}\) is
  \[
    \begin{split}
      &\abs{\set{\fun{\isets}{\minor{M_1 \oplus M_2}{\pqty{\pqty{X_1 \cup X_2} \setminus \pqty{Y_1 \cup Y_2}}}{Y_1 \cup Y_2}}}{\text{\(Y_1 \subseteq X_1\) and \(X_2 \subseteq X_2\)}}} \\
      &= \abs*{%
        \set*{\Bqty*{%
          \begin{array}{@{}l@{}}
            I_1 \cup I_2 \subseteq \pqty{\fun{\gset}{M_1} \cup \fun{\gset}{M_2}} \setminus \pqty{X_1 \cup X_2} \\
            \mathrel{:} \text{\(\restrict{M_1}{Y_1}\) has a basis \(B_1\) such that \(I_1 \cup B_1 \in \fun{\isets}{M_1}\) and} \\
            \phantom{\mathrel{:} \text{}} \text{\(\restrict{M_2}{Y_2}\) has a basis \(B_2\) such that \(I_2 \cup B_2 \in \fun{\isets}{M_2}\)}
          \end{array}%
        }}{\text{\(Y_1 \subseteq X_1\) and \(Y_2 \subseteq X_2\)}}
      } \\
      &= \abs*{%
        \begin{array}{@{}c@{}}
          \set{\set{I_1 \subseteq \fun{\gset}{M_1} \setminus X_1}{\text{\(\restrict{M_1}{Y_1}\) has a basis \(B_1\) such that \(I_1 \cup B_1 \in \fun{\isets}{M_1}\)}}}{Y_1 \subseteq X_1} \\
          \times \\
          \set{\set{I_2 \subseteq \fun{\gset}{M_2} \setminus X_2}{\text{\(\restrict{M_2}{Y_2}\) has a basis \(B_2\) such that \(I_2 \cup B_2 \in \fun{\isets}{M_2}\)}}}{Y_2 \subseteq X_2}
        \end{array}%
      } \\
      &= m_1 m_2\text{.}\qedhere
    \end{split}
  \]
\end{proof}

By this lemma, we can show the result for partition matroids (see \cref{eg:partition}):

\begin{theorem}\label{thm:width_partition}
  Let \(M\) be a partition matroid on the ground set \(\gset\).
  The number of minors of \(M\) on \(\eivar\) is at most \(2^{\confun_{M}\pqty{\ei}}\).
\end{theorem}

\begin{proof}
  There exist uniform matroids \(U_{r_1, n_1}, \ldots, U_{r_k, n_k}\) such that \(M = U_{r_1, n_1} \oplus \cdots \oplus U_{r_k, n_k}\).
  Let \(G_j\) be the ground set of size \(n_j\) of \(U_{r_j, n_j}\) for \(j \in \Bqty{1, \ldots, k}\).
  By \Cref{thm:width_uniform,lem:direct_sum_width}, the number \(m\) of minors on \(\eivar\) satisfies that
  \[
    \begin{split}
      \log_2{m}
      &= \sum_{j = 1}^k \log_2{\pqty{\confun_{U_{r_j, n_j}}\pqty{\ei \cap G_j} + 1}} \\
      &= \sum_{j = 1}^k \log_2{\pqty{\min\Bqty{\abs{\ei \cap G_j}, r_j} + \min\Bqty{n_j - \abs{\ei \cap G_j}, r_j} - r_j + 1}} \\
      &\leq \sum_{j = 1}^k \pqty{\min\Bqty{\abs{\ei \cap G_j}, r_j} + \min\Bqty{n_j - \abs{\ei \cap G_j}, r_j} - r_j} \\
      &= \rankm_{M}\pqty{\ei} + \rankm_{M}\pqty{\eivar} - \rankm\pqty{M} = \confun_{M}\pqty{\ei}\text{.}
    \end{split}
  \]
  Here we use the fact that every \(n \in \mathbb{N}\) satisfies \(n + 1 \leq 2^n\).
\end{proof}

Now we introduce the Gale order.
For a set \(E\) and \(k \in \Bqty{0, \ldots, \abs{E}}\), the \emph{Gale order} \(\galeorder{\preceq}\) on a family of \(k\)-element subsets of \(E\) with respect to a total order \(\preceq\) on \(E\) is defined as follows.
For \(k\)-element subset \(X \subseteq E\), let \(\sequence\pqty{X}\) be the sequence of elements in \(X\) arranged in ascending order with respect to \(\preceq\) and \({\sequence\pqty{X}}_j\) be the \(j\)th element of \(\sequence\pqty{X}\).
Then, for \(k\)-element subsets \(X_1, X_2 \subseteq E\), define \(\galeorder{\preceq}\) such that \(X_1 \galeorder{\preceq} X_2\) if and only if \({\sequence\pqty{X_1}}_j \preceq {\sequence\pqty{X_2}}_j\) for all \(j \in \Bqty{1, \ldots, k}\).
For a matroid \(M\) and a total order \(\preceq\) on \(\fun{\gset}{M}\), it is known that there uniquely exists a maximal basis \(\galebasis\pqty{M, \preceq} \in \fun{\bases}{M}\) with respect to \(\galeorder{\preceq}\)~\cite[Theorem 3]{Gale1968}, and hence \(B \galeorder{\preceq} \galebasis\pqty{M, \preceq}\) for all bases \(B \in \fun{\bases}{M}\).
This is called the \emph{Gale basis}.
In this paper, we denote by \(\galebasis\pqty{M, \preceq}\) the Gale basis of a matroid \(M\) with respect to a total order \(\preceq\).

For a nested matroid (see \cref{eg:nested}) \(M\), there exists a total order \(\leftjustified\) on \(\fun{\gset}{M}\) such that, for all \(B \subseteq \fun{\gset}{M}\) with \(\abs{B} = \rankm\pqty{M}\), the set \(B\) is a basis of \(M\) if and only if \(B \galeorder{\leftjustified} \galebasis\pqty{M, \leftjustified}\) (see \cite[Section 2.2]{bonin2024characterization}).
In this paper, we call such an order \(\leftjustified\) a \emph{left-justified order} of \(M\).
Any left-justified order preserves its properties in the minors of \(M\).
Here, for a totally ordered set \(\pqty{E, \preceq}\) and \(X \subseteq E\), we denote by \(\pqty{X, \preceq}\) the totally ordered set induced from \(\pqty{E, \preceq}\).

\begin{lemma}\label{lem:order_in_nested}
  Let \(M\) be a nested matroid and \(\leftjustified\) be a left-justified order of \(M\).
  Then, for all disjoint \(X, Y \subseteq \fun{\gset}{M}\) and all \(B \subseteq \fun{\gset}{M} \setminus \pqty{X \cup Y}\) with \(\abs{B} = \rankm\pqty{\minor{M}{X}{Y}}\), the set \(B\) is a basis of \(\minor{M}{X}{Y}\) if and only if \(B \galeorder{\leftjustified} \galebasis\pqty{\minor{M}{X}{Y}, \leftjustified}\).
\end{lemma}

\begin{proof}
  Since every minor of nested matroids is nested~\cite[Lemma 6 and 9]{Oxley_Prendergast_Row_1982}, it suffices to show that both \cref{item:contraction_in_lem}~\(\contraction{M}{\Bqty{e}}\) and \cref{item:deletion_in_lem}~\(\deletion{M}{\Bqty{e}}\) satisfy the condition for any \(e \in \fun{\gset}{M}\).
  Note that we only need to show the necessity since the sufficiency is achieved directly by the properties of the Gale basis.
  \begin{enumerate}[label=(\roman*)]
    \item\label{item:contraction_in_lem}
      If \(e\) is a loop of \(M\), then we have \(\fun{\bases}{\contraction{M}{\Bqty{e}}} = \fun{\bases}{M}\).
      If not, then \(\galebasis\pqty{\contraction{M}{\Bqty{e}}, \leftjustified} \cup \Bqty{e}\) is a basis of \(M\) since \(\Bqty{e}\) is a basis of \(\restrict{M}{\Bqty{e}}\).
      Thus, for all \(B \subseteq \fun{\gset}{M} \setminus \Bqty{e}\) with \(B \galeorder{\leftjustified} \galebasis\pqty{\contraction{M}{\Bqty{e}}, \leftjustified}\), we have \(B \cup \Bqty{e} \galeorder{\leftjustified} \galebasis\pqty{\contraction{M}{\Bqty{e}}, \leftjustified} \cup \Bqty{e} \galeorder{\leftjustified} \galebasis\pqty{M, \leftjustified}\), implying that \(B \cup \Bqty{e}\) is a basis of \(M\).
    \item\label{item:deletion_in_lem}
      If \(e\) is a coloop of \(M\), then \(\deletion{M}{\Bqty{e}} = \contraction{M}{\Bqty{e}}\) (see \cite[Corollary 3.1.24]{10.1093/acprof:oso/9780198566946.001.0001}).
      Hence the necessity can be shown by \cref{item:contraction_in_lem}.
      If not, then the rank of \(\deletion{M}{\Bqty{e}}\) is the same as that of \(M\).
      Since the collection of bases of \(\deletion{M}{\Bqty{e}}\) is the family of all maximal elements of \(\set{B \setminus \Bqty{e}}{B \in \fun{\bases}{M}}\), we have \(\fun{\bases}{\deletion{M}{\Bqty{e}}} = \set{B \in \fun{\bases}{M}}{e \notin B}\).
      Thus, for all \(B \subseteq \fun{\gset}{M} \setminus \Bqty{e}\) with \(B \galeorder{\leftjustified} \galebasis\pqty{\deletion{M}{\Bqty{e}}, \leftjustified}\), we have \(B \galeorder{\leftjustified} \galebasis\pqty{\deletion{M}{\Bqty{e}}, \leftjustified} \galeorder{\leftjustified} \galebasis\pqty{M, \preceq}\), implying that \(B\) is a basis of \(M\).\qedhere
  \end{enumerate}
\end{proof}

This lemma allows us to test the equivalence of minors solely with respect to their Gale bases.
We can show the same upper bound for nested matroids as for partition matroids:

\begin{theorem}\label{thm:width_nested}
  Let \(M\) be a nested matroid on the ground set \(\gset\).
  The number of minors of \(M\) on \(\eivar\) is at most \(2^{\confun_{M}\pqty{\ei}}\).
\end{theorem}

\begin{proof}
  Let \(\leftjustified\) be a left-justified order of \(M\).
  By \Cref{lem:order_in_nested}, the number of minors of \(M\) on \(\eivar\) coincides with the size of \(\set{\galebasis\pqty{\minori{X}, \leftjustified}}{X \subseteq \ei}\).

  Now we show that \(\galebasis\pqty{\contraction{M}{\Bqty{e}}, \leftjustified} \subseteq \galebasis\pqty{\deletion{M}{\Bqty{e}}, \leftjustified}\) for all \(e \in \gset\).
  If \(e\) is either a loop or coloop of \(M\), then we have \(\deletion{M}{\Bqty{e}} = \contraction{M}{\Bqty{e}}\).
  Hence assume that \(e\) is neither a loop nor a coloop.
  Let \(\galebasis\pqty{M, \leftjustified} \coloneqq \Bqty{g_1, \ldots, g_{\rankm\pqty{M}}}\) such that \(g_1 \leftjustified* \cdots \leftjustified* g_{\rankm\pqty{M}}\).
  Since \(e\) is not a loop, we can define \(\ell \in \Bqty{1, \ldots, \rankm\pqty{M}}\) as \(\min\set{j}{e \leftjustified g_j}\).
  We find the Gale bases of \cref{item:contraction_in_thm}~\(\contraction{M}{\Bqty{e}}\) and \cref{item:deletion_in_thm}~\(\deletion{M}{\Bqty{e}}\) with \(\ell\).
  Note that \(\rankm\pqty{M} = \rankm\pqty{\deletion{M}{\Bqty{e}}} = \rankm\pqty{\contraction{M}{\Bqty{e}}} + 1\).
  Here, for \(X \subseteq \gset\), we denote by \(\sequence\pqty{X}\) the sequence of elements in \(X\) arranged in ascending order with respect to \(\leftjustified\) and by \({\sequence\pqty{X}}_j\) the \(j\)th element of \(\sequence\pqty{X}\).
  \begin{enumerate}[label=(\roman*)]
    \item\label{item:contraction_in_thm}
      Since \(e \leftjustified g_\ell\), we have \(\pqty{\galebasis\pqty{M, \leftjustified} \setminus \Bqty{g_\ell}} \cup \Bqty{e} \galeorder{\leftjustified} \galebasis\pqty{M, \leftjustified}\).
      Hence \(\galebasis\pqty{M, \leftjustified} \setminus \Bqty{g_\ell}\) is a basis of \(\contraction{M}{\Bqty{e}}\).
      Assume that \(B\) is a basis of \(\contraction{M}{\Bqty{e}}\).
      We can show that \(B \cup \Bqty{e} \galeorder{\leftjustified} \galebasis\pqty{M, \leftjustified}\), and there exists exactly one \(j \in \Bqty{1, \ldots, \rankm\pqty{M}}\) such that \({\sequence\pqty{B \cup \Bqty{e}}}_j = e\).
      If \(j < \ell\), then \(e \leftjustified g_j\) contradicts the choice of \(\ell\).
      If \(j = \ell\), then considering the sequences \(\sequence\pqty{B \cup \Bqty{e}}\) and \(\sequence\pqty{\galebasis\pqty{M, \leftjustified}}\) without the \(j\)th element, we have \(B \galeorder{\leftjustified} \galebasis\pqty{M, \leftjustified} \setminus \Bqty{g_\ell}\).
      If \(j > \ell\), then \({\sequence\pqty{B}}_\ell \leftjustified* \cdots \leftjustified* {\sequence\pqty{B}}_{j - 1} \leftjustified* e \leftjustified g_\ell \leftjustified* \cdots \leftjustified* g_j\) implies that \(\pqty{{\sequence\pqty{B}}_\ell, \ldots, {\sequence\pqty{B}}_{j - 1}} \galeorder{\leftjustified} \pqty{g_{\ell + 1}, \ldots, g_j}\).
      Thus \(\galebasis\pqty{M, \leftjustified} \setminus \Bqty{g_\ell}\) is the Gale basis of \(\contraction{M}{\Bqty{e}}\).
    \item\label{item:deletion_in_thm}
      Since \(\fun{\bases}{\deletion{M}{\Bqty{e}}} = \set{B \in \fun{\bases}{M}}{e \notin B}\), if \(e \leftjustified* g_\ell\), then \(\galebasis\pqty{M, \leftjustified}\) is the Gale basis of \(\deletion{M}{\Bqty{e}}\).
      Assume that \(e = g_\ell\) below.\newline
      Let \(\gset\) be a set \(\Bqty{e_1, \ldots, e_n}\) such that \(e_1 \leftjustified* \cdots \leftjustified* e_n\) and \(m \in \Bqty{\ell, \ldots, n}\) be a positive integer such that \(e_m = g_\ell\).
      If \(m = \ell\), then \(\galebasis\pqty{M, \leftjustified} = \Bqty{e_1, \ldots, e_m, g_{\ell + 1}, \ldots, g_{\rankm\pqty{M}}}\) implies that every basis of \(M\) contains \(e_m\), which contradicts that \(e = e_m\) is not a coloop.
      Hence we have \(m > \ell\).
      Now let \(k \coloneqq \min\set{j \in \Bqty{1, \ldots, \ell}}{g_j = e_{m - \pqty{\ell - j}}}\).
      Then, for \(j \in \Bqty{1, \ldots, \rankm\pqty{M}}\), we can define \(h_j\) by
      \[
        h_j \coloneqq \begin{cases}
          e_{m - \pqty{\ell - j} - 1} & \text{if \(j \in \Bqty{k, \ldots, \ell}\),} \\
          g_j & \text{otherwise.} \\
        \end{cases}
      \]
      By the definition of \(k\), we have \(h_1 \leftjustified* \cdots \leftjustified* h_{\rankm\pqty{M}}\).
      Let \(H \coloneqq \Bqty{h_1, \ldots, h_{\rankm\pqty{M}}}\), and then we show that \(H\) is the Gale basis of \(\deletion{M}{\Bqty{e}}\).
      Note that \(e = e_m \notin H\).\newline
      Since \(h_j \leftjustified g_j\) for all \(j \in \Bqty{1, \ldots, \rankm\pqty{M}}\), the set \(H\) is a basis of \(\deletion{M}{\Bqty{e}}\).
      Now suppose that there exists a basis \(B \in \fun{\bases}{\deletion{M}{\Bqty{e}}}\) such that \(H \galeorder{\leftjustified*} B\).
      Since \(B\) is also a basis of \(M\), every \(j \in \Bqty{1, \ldots, \rankm\pqty{M}}\) satisfies that \({\sequence\pqty{B}}_j \leftjustified g_j\).
      Hence there exists \(x \in \Bqty{k, \ldots, \ell - 1}\) such that \(h_x = e_{m - \pqty{\ell - x} - 1} \leftjustified* {\sequence\pqty{B}}_x = g_x = e_{m - \pqty{\ell - x}}\).
      By the definition of \(H\), we can show that \(h_j \leftjustified* {\sequence\pqty{B}}_j = g_j\) for all \(j \in \Bqty{x, \ldots, \ell - 1}\).
      In particular, for \(j = \ell - 1\), we deduce that \({\sequence\pqty{B}}_{\ell - 1} = e_{m - 1}\).
      However, since \(e_m \notin B\), there is no element \({\sequence\pqty{B}}_\ell\) such that \({\sequence\pqty{B}}_{\ell - 1} = e_{m - 1} \leftjustified* {\sequence\pqty{B}}_\ell \leftjustified g_\ell = e_m\); a contradiction.
      Thus \(H\) is the Gale basis of \(\deletion{M}{\Bqty{e}}\).
      Note that \(H\) can be expressed as \(\pqty{\galebasis\pqty{M, \leftjustified} \setminus \Bqty{g_\ell}} \cup \Bqty{e_{m - \pqty{\ell - k} - 1}}\).
  \end{enumerate}

  In summary, \(\galebasis\pqty{M, \leftjustified} \setminus \Bqty{g_\ell}\) is the Gale basis of \(\contraction{M}{\Bqty{e}}\), and \(\galebasis\pqty{M, \leftjustified}\) or \(\pqty{\galebasis\pqty{M, \leftjustified} \setminus \Bqty{g_\ell}} \cup \Bqty{e_{m - \pqty{\ell - k} - 1}}\) is the Gale basis of \(\deletion{M}{\Bqty{e}}\).
  This implies that \(\galebasis\pqty{\contraction{M}{\Bqty{e}}, \leftjustified} \subseteq \galebasis\pqty{\deletion{M}{\Bqty{e}}, \leftjustified}\) for all \(e \in \gset\).
  Then it follows that \(\galebasis\pqty{\contraction{M}{\ei}, \leftjustified} \subseteq \galebasis\pqty{\minori{X}, \leftjustified} \subseteq \galebasis\pqty{\deletion{M}{\ei}, \leftjustified}\) for all \(X \subseteq \ei\).
  Therefore, the number of minors on \(\eivar\) is at most
  \[
    \begin{split}
      2^{\abs{\galebasis\pqty{\deletion{M}{\ei}, \leftjustified} \setminus \galebasis\pqty{\contraction{M}{\ei}, \leftjustified}}}
      &= 2^{\abs{\galebasis\pqty{\deletion{M}{\ei}, \leftjustified}} - \abs{\galebasis\pqty{\contraction{M}{\ei}, \leftjustified}}} \\
      &= 2^{\rankm_{M}\pqty{\eivar} - \pqty{\rankm\pqty{M} - \rankm_{M}\pqty{\ei}}} = 2^{\confun_{M}\pqty{\ei}}\text{.}\qedhere
    \end{split}
  \]
\end{proof}

So far, we have discussed classes of matroids where the width can be bounded by a value depending only on \(\confun\pqty{\ei}\).
However, some classes can not be bounded.
Before showing such classes, we can show the following:

\begin{lemma}\label{lem:exact_width_zdd}
  The \(i\)th width of the ZDD \(\zdd{\fun{\isets}{M}}[\gset][\preceq]\) is equal to the number of minors of \(M\) on \(\eivar\) for which \(e_i\) is not a loop.
\end{lemma}

\begin{proof}
  As mentioned above, the sub-ZDDs of \(\zdd{\fun{\isets}{M}}[\gset][\preceq]\) rooted by a non-terminal node with label \(e_i\) represent the collections of independent sets of some minors of \(M\) on \(\eivar\).
  Conversely, there is a sub-ZDD representing the collection of independent sets for every minor on \(\eivar\).
  If \(e_i\) is not a loop, that is, \(\Bqty{e_i}\) is independent in a certain minor, then the corresponding sub-ZDD contains non-terminal nodes with label \(e_i\), implying that the root label of this sub-ZDD is \(e_i\).
  On the other hand, if \(e_i\) is a loop, then the label \(e_i\) does not appear in non-terminal nodes of the sub-ZDD.
\end{proof}

We also show that the \(i\)th width of the BDD/ZDD for a transversal (see \cref{eg:nested}) or laminar (see \cref{eg:laminar}) matroid cannot be bounded by a value depending only on \(\confun\pqty{\ei}\).

\begin{theorem}\label{thm:width_transversal}
  There does not exist a function \(f \colon \mathbb{N} \to \mathbb{N}\) such that, for any transversal matroid \(M\) on the ground set \(\gset\), total order \(\preceq\) on \(\gset\), and \(i \in \Bqty{0, \ldots, \abs{\gset} - 1}\), the \(i\)th width of the ZDD \(\zdd{\fun{\isets}{M}}[\gset][\preceq]\) is at most \(\fun{f}{\confun_{M}\pqty{\ei}}\).
\end{theorem}

\begin{proof}
  Since the class of transversal matroids is not strongly pigeonhole~\cite[Proposition 20]{Funk2023}, there exists a positive integer \(\mu\) such that, for all positive integers \(\fun{g}{\mu}\) depending only on \(\mu\), there are a transversal matroid \(M^\prime\) and \(S \subseteq \fun{\gset}{M^\prime}\) with \(\confun_{M^\prime}\pqty{S} \leq \mu\) satisfying that the size of the quotient set by \(\sim_{M^\prime, S}\) is greater than \(\fun{g}{\mu}\).
  Suppose that \(\fun{f}{\mu}\) is an upper bound on the width.
  There exist a transversal matroid \(M^\prime\) on the ground set \(\gset^{\pqty{1}}\) and \(S \subseteq \gset^{\pqty{1}}\) with \(\confun_{M^\prime}\pqty{S} \leq \mu\) such that the size of the quotient set by \(\sim_{M^\prime, S}\) is at least \(\fun{f}{\mu} + 2\).
  Let \(d \coloneqq \mu - \confun_{M^\prime}\pqty{S}\).
  Then, since the direct sum \(M^\prime \oplus U_{1, 1} \oplus U_{d, 2d}\) of the matroid \(M^\prime\), a free matroid \(U_{1, 1}\) on the ground set \(\gset^{\pqty{2}}\), and a uniform matroid \(U_{d, 2d}\) on the ground set \(\gset^{\pqty{3}} \cup \gset^{\pqty{4}}\) is transversal (see \cite[Proposition 4.2.11]{10.1093/acprof:oso/9780198566946.001.0001}), it suffices to show that the \(\pqty{m + d}\)-th width of the ZDD \(\zdd{\fun{\isets}{M^\prime \oplus U_{1, 1} \oplus U_{d, 2d}}}[\gset^{\pqty{1}} \cup \gset^{\pqty{2}} \cup \gset^{\pqty{3}} \cup \gset^{\pqty{4}}][\preceq]\) is greater than \(\fun{f}{\mu}\).
  Here let \(S \coloneqq \Bqty{s_1, \ldots, s_m}\), \(\gset^{\pqty{1}} \setminus S \coloneqq \Bqty{\overline{s}_1, \ldots, \overline{s}_{n - m}}\), \(\gset^{\pqty{2}} \coloneqq \Bqty{e_{2, 1}}\), \(\gset^{\pqty{3}} \coloneqq \Bqty{e_{3, 1}, \ldots, e_{3, d}}\), \(\gset^{\pqty{4}} \coloneqq \Bqty{e_{4, 1}, \ldots, e_{4, d}}\), and \(\preceq\) be a total order on \(\gset^{\pqty{1}} \cup \gset^{\pqty{2}} \cup \gset^{\pqty{3}} \cup \gset^{\pqty{4}}\) such that \(s_1 \preceq \cdots \preceq s_m \preceq e_{3, 1} \preceq \cdots \preceq e_{3, d} \preceq e_{2, 1} \preceq \overline{s}_1 \preceq \cdots \preceq \overline{s}_{n - m} \preceq e_{4, 1} \preceq \cdots \preceq e_{4, d}\).

  By \cref{thm:width_uniform,lem:direct_sum_width}, the number of minors of \(M^\prime \oplus U_{1, 1} \oplus U_{d, 2d}\) on \(\pqty{\gset^{\pqty{1}} \setminus S} \cup \gset^{\pqty{2}} \cup \gset^{\pqty{4}}\) is the product of the number of minors of \(M^\prime\) on \(\gset^{\pqty{1}} \setminus S\), of \(U_{1, 1}\) on \(\gset^{\pqty{2}}\), and of \(U_{d, 2d}\) on \(\gset^{\pqty{4}}\).
  Moreover, \(\gset^{\pqty{2}}\) is independent in all minors of \(M^\prime \oplus U_{1, 1} \oplus U_{d, 2d}\) on \(\pqty{\gset^{\pqty{1}} \setminus S} \cup \gset^{\pqty{2}} \cup \gset^{\pqty{4}}\).
  Thus, by \cref{lem:exact_width_zdd}, the \(\pqty{m + d}\)-th width of \(\zdd{\fun{\isets}{M^\prime \oplus U_{1, 1} \oplus U_{d, 2d}}}[\gset^{\pqty{1}} \cup \gset^{\pqty{2}} \cup \gset^{\pqty{3}} \cup \gset^{\pqty{4}}][\preceq]\) is at least \(\pqty{\fun{f}{\mu} + 1} \cdot 1 \cdot \pqty{\confun_{U_{d, 2d}}\pqty{\gset^{\pqty{3}}} + 1} = \pqty{\fun{f}{\mu} + 1}\pqty{d + 1}\).
  However, we have
  \[
    \begin{split}
      \confun_{M^\prime \oplus U_{1, 1} \oplus U_{d, 2d}}\pqty{S \cup \gset^{\pqty{3}}}
      &= \pqty{\rankm_{M^\prime}\pqty{S} + d} + \pqty{\rankm_{M^\prime}\pqty{\gset^{\pqty{1}} \setminus S} + 1 + d} - \pqty{\rankm\pqty{M^\prime} + 1 + d} \\
      &= \confun_{M^\prime}\pqty{S} + d = \mu\text{;}
    \end{split}
  \]
  a contradiction.
\end{proof}

\begin{theorem}\label{thm:width_laminar}
  There does not exist a function \(f \colon \mathbb{N} \to \mathbb{N}\) such that, for any laminar matroid \(M\) on the ground set \(\gset\), total order \(\preceq\) on \(\gset\), and \(i \in \Bqty{0, \ldots, \abs{\gset} - 1}\), the \(i\)th width of the ZDD \(\zdd{\fun{\isets}{M}}[\gset][\preceq]\) is at most \(\fun{f}{\confun_{M}\pqty{\ei}}\).
\end{theorem}

\begin{proof}
  Suppose that there exists a constant upper bound \(\fun{f}{2}\).
  Let \(\gset\) be a set \(\Bqty{e_1, \ldots, e_{2\pqty{\fun{f}{2} + 2}}}\), \(\preceq\) be a total order on \(\gset\) such that \(e_1 \prec e_3 \prec \cdots \prec e_{2\fun{f}{2} + 3} \prec e_2 \prec e_4 \prec \cdots \prec e_{2\pqty{\fun{f}{2} + 2}}\), and \(M\) be a laminar matroid on \(\gset\) whose collection of independent sets is
  \[
    \set{I \subseteq \gset}{\text{\(\abs{I \cap \Bqty{e_{2j - 1}, e_{2j}}} \leq 1\) and \(\abs{I \cap \Bqty{e_1, \ldots, e_{2j}}} \leq 2\) for all \(j \in \Bqty{1, \ldots, \fun{f}{2} + 2}\)}}\text{.}
  \]
  Then, for \(j \in \Bqty{1, 3, \ldots, 2\fun{f}{2} + 3}\), the collection \(\fun{\isets}{\minori{\Bqty{e_j}}}\) equals \(\Bqty{\emptyset, \Bqty{e_2}, \Bqty{e_4}, \ldots,\\ \Bqty{e_{2\pqty{\fun{f}{2} + 2}}}} \setminus \Bqty{e_{j + 1}}\).
  Thus, by \cref{lem:exact_width_zdd}, the \(\pqty{\fun{f}{2} + 2}\)-th width of \(\zdd{\fun{\isets}{M}}[\gset][\preceq]\) is at least \(\fun{f}{2} + 1\).
  On the other hands, we have \(\confun_{M}\pqty{\ei} = 2 + 2 - 2 = 2\); a contradiction.
\end{proof}

\Cref{thm:width_uniform,cor:width_free,thm:width_partition,thm:width_nested} collectively establish \cref{thm:intro_width}.
These results, along with \cref{thm:width_transversal,thm:width_laminar}, are illustrated as in \cref{fig:width_repr_matroid}.
Note that the relations between the classes of matroids can be structured with reference to the following three examples.

\begin{myexample}[Transversal matroid that is not laminar]
  For \(r \in \mathbb{N} \setminus \Bqty{0, 1, 2}\), let \(\gset_r \coloneqq \Bqty{e_1, \ldots, e_{2r - 1}}\) and
  \[
    \isets_r \coloneqq \set{I \subseteq \gset}{\text{\(\abs{I} \leq r\), \(I \neq \Bqty{e_1, \ldots, e_r}\), and \(I \neq \Bqty{e_r, \ldots, e_{2r - 1}}\)}}\text{.}
  \]
  Then \(Y_r\) is a transversal matroid, where \(Y_r \coloneqq \pqty{\gset_r, \isets_r}\).
  Indeed, \(Y_r\) has the presentation \(\Bqty{\Bqty{e_1, \ldots, e_{r - 1}}, \Bqty{e_2, \ldots, e_{r + 1}}, \Bqty{e_3, \ldots, e_{r + 2}} \ldots, \Bqty{e_{r - 1}, \ldots, e_{2r - 2}}, \Bqty{e_{r + 1}, \ldots, e_{2r - 1}}}\).
  On the other hands, \(Y_r\) is not laminar~\cite[Theorem 1.2]{FIFE2017206}.
\end{myexample}

\begin{myexample}[Laminar matroid that is not transversal]
  For the undirected graph \(G\) in \Cref{fig:transversal_not_laminar}, the cycle matroid \(\fun{M}{G}\) of \(G\) is not transversal (see \cite[Example 1.6.3]{10.1093/acprof:oso/9780198566946.001.0001}).
  On the other hand, the collection of independent sets of \(\fun{M}{G}\) can be represented as
  \[
    \set*{I \subseteq \Bqty{e_1, \ldots, e_6}}{%
      \begin{array}{@{}l@{}}
        \text{\(\abs{I \cap \Bqty{e_1, e_2}} \leq 1\), \(\abs{I \cap \Bqty{e_3, e_4}} \leq 1\), \(\abs{I \cap \Bqty{e_5, e_6}} \leq 1\),} \\
        \text{and \(\abs{I \cap \Bqty{e_1, \ldots, e_6}} \leq 2\)}
      \end{array}%
    }\text{.}
  \]
  Thus \(\fun{M}{G}\) is laminar.
\end{myexample}

\begin{myexample}[Matroid that is not nested but both transversal and laminar]
  Let \(\gset \coloneqq \Bqty{e_1, \ldots, e_6}\) and
  \[
    \isets \coloneqq \set{I \subseteq \gset}{\text{\(\abs{I \cap \Bqty{e_1, e_2, e_3}} \leq 2\), \(\abs{I \cap \Bqty{e_4, e_5, e_6}} \leq 2\), and \(\abs{I \cap \gset} \leq 3\)}}\text{.}
  \]
  Then \(N^3\) is a laminar matroid, where \(N^3 \coloneqq \pqty{\gset, \isets}\).
  Moreover, \(N^3\) is transversal because it has the presentation \(\pqty{\Bqty{e_1, e_2, e_3}, \Bqty{e_4, e_5, e_6}, \gset}\).
  On the other hands, \(N^3\) is not nested~\cite[Theorem 13]{Oxley_Prendergast_Row_1982}.
\end{myexample}

\subsection{Upper bounds on widths in good orders}\label{sec:width_with_pathwidth}

In the previous section, we considered any total order \(\preceq\) on \(\gset\).
However, certain orders may result in smaller widths of BDDs/ZDDs.
For instance, as discussed in \cref{sec:related_work}, there exists a total order \(\preceq^\ast\) on \(E\) such that the width of the BDD/ZDD for the cycle matroid of an undirected graph \(G\) with respect to \(\pqty{E, \preceq^\ast}\) is at most the \(\operatorname{lw}\pqty{G}\)-th Bell number, where \(\operatorname{lw}\pqty{G}\) is the linear-width~\cite[F9]{thomas1996} of \(G\).

In this section, we use the pathwidth of matroids as a parameter of the width and prove \cref{thm:intro_pathwidth}.
Recall from \cref{sec:introduction} that the pathwidth \(\pathwidth\pqty{M}\) of a matroid \(M\) on the \(n\)-element set \(\gset\) is defined by
\[
  \pathwidth\pqty{M} \coloneqq \min\set{\max\set{\confun_{M}\pqty{\ei}}{i \in \Bqty{1, \ldots, n}}}{\text{\(\preceq\) is a total order on \(\gset\)}}\text{.}
\]
Note that we sometimes use the following definition since \(\confun_{M}\pqty{\emptyset} = \confun_{M}\pqty{\gset} = 0\):
\[
  \pathwidth\pqty{M} \coloneqq \min\set{\max\set{\confun_{M}\pqty{\ei}}{i \in \Bqty{0, \ldots, n - 1}}}{\text{\(\preceq\) is a total order on \(\gset\)}}\text{.}
\]
Recall also from \cref{sec:introduction} that we refer to the four types of BDDs/ZDDs representing either a collection of the independent sets or bases of a matroid as BDDs/ZDDs for a matroid.

We first provide the upper bound for uniform matroids.

\begin{theorem}
  Let \(M\) be a uniform matroid on the ground set \(\gset\).
  For all total orders \(\preceq\) on \(\gset\), the width of the BDDs/ZDDs for \(M\) with respect to \(\pqty{\gset, \preceq}\) is at most \(\pathwidth\pqty{M} + 1\).
\end{theorem}

\begin{proof}
  The pathwidth of \(M\) is \(\max\set{\min\Bqty{i, \rankm\pqty{M}} + \min\Bqty{n - i, \rankm\pqty{M}} - \rankm\pqty{M}}{i \in \Bqty{0, \ldots, n - 1}}\).
  By \cref{thm:width_uniform}, the width of the BDD/ZDD for \(M\) with respect to \(\pqty{\gset, \preceq}\) is at most \(\max\set{\confun_{M}\pqty{\ei} + 1}{i \in \Bqty{0, \ldots, n - 1}} = \pathwidth\pqty{M} + 1\).
\end{proof}

The pathwidth of a free matroid is clearly zero.
Thus we obtain the following since all free matroids are uniform:

\begin{corollary}
  Let \(M\) be a free matroid on the ground set \(\gset\).
  For all total orders \(\preceq\) on \(\gset\), the width of the BDDs/ZDDs for \(M\) with respect to \(\pqty{\gset, \preceq}\) is at most one.
\end{corollary}

The last theorem and corollary consider any total order on \(\gset\).
The following theorems show that a certain total order on \(\gset\) can ensure that the width of BDDs/ZDDs is at most \(\pathwidth\pqty{M} + 1\).

\begin{theorem}\label{thm:pathwidth_partition}
  Let \(M\) be a partition matroid on the ground set \(\gset\).
  There exists a total order \(\preceq^\ast\) on \(\gset\) such that the width of the BDD/ZDD for \(M\) with respect to \(\pqty{\gset, \preceq^\ast}\) is at most \(\pathwidth\pqty{M} + 1\).
\end{theorem}

\begin{proof}
  It suffices to consider the case when \(\rankm\pqty{M} \geq 1\).
  There exist uniform matroids \(U_{r_1, n_1}, \ldots, U_{r_k, n_k}\) such that \(M = U_{r_1, n_1} \oplus \cdots \oplus U_{r_k, n_k}\), where \(k \geq 1\) and \(n = \sum_{j = 1}^k n_j\).
  For \(j \in \Bqty{1, \ldots, k}\), let \(\gset^{\pqty{j}} \coloneqq \Bqty{e_{j, 1}, \ldots, e_{j, n_j}}\) as the ground set of \(U_{r_j, n_j}\) and \(p \coloneqq \argmax_{j \in \Bqty{1, \ldots, k}} \Bqty{n_j - r_j, r_j}\).

  First, we show the lower bound on the pathwidth of \(M\).
  For any total order \(\preceq\) on \(\gset\), there exists \(i \in \Bqty{1, \ldots, n}\) such that \(\abs{\ei \cap \gset^{\pqty{p}}} = r_p\).
  Thus we have \(\pathwidth\pqty{M} \geq r_p + \min\Bqty{n_p - r_p, r_p} - r_p = \min\Bqty{n_p - r_p, r_p}\).

  Next, we can define the total order \(\preceq^\ast\) on \(\gset\) such that \(e_{1, 1} \prec^\ast \cdots \prec^\ast e_{1, n_1} \prec^\ast e_{2, 1} \prec^\ast \cdots \prec^\ast e_{2, n_2} \prec^\ast \cdots \prec^\ast e_{k, 1} \prec^\ast \cdots \prec^\ast e_{k, n_k}\).
  By \cref{thm:width_uniform,lem:direct_sum_width}, the width of the BDD/ZDD for \(M\) with respect to \(\pqty{\gset, \preceq^\ast}\) is at most
  \[
    \begin{split}
      &\max\set{\confun_{U_{r_j, n_j}}\pqty{\Bqty{e_{j, 1}, \ldots, e_{j, i}}} + 1}{\text{\(i \in \Bqty{0, \ldots, n_j - 1}\) and \(j \in \Bqty{1, \ldots, k}\)}} \\
      &= \max\set{\min\Bqty{i, r_j} + \min\Bqty{n_j - i, r_j} - r_j + 1}{\text{\(i \in \Bqty{0, \ldots, n_j - 1}\) and \(j \in \Bqty{1, \ldots, k}\)}} \\
      &= \max\set{\min\Bqty{i - r_j, 0} + \min\Bqty{n_j - i, r_j} + 1}{\text{\(i \in \Bqty{0, \ldots, n_j - 1}\) and \(j \in \Bqty{1, \ldots, k}\)}} \\
      &= \max\set{\min\Bqty{n_j - r_j, r_j}}{j \in \Bqty{1, \ldots, k}} + 1
      = \min\Bqty{n_p - r_p, r_p} + 1\text{.}\qedhere
    \end{split}
  \]
\end{proof}

\begin{theorem}\label{thm:pathwidth_nested}
  Let \(M\) be a nested matroid on the ground set \(\gset\) and \(\leftjustified\) be a left-justified order of \(M\).
  Then the width of the BDD/ZDD for \(M\) with respect to \(\pqty{\gset, \leftjustified}\) is at most \(\pathwidth\pqty{M} + 1\).
\end{theorem}

\begin{proof}
  It suffices to consider the case when \(\rankm\pqty{M} \geq 1\).
  Let \(\gset \coloneqq \Bqty{e_1, \ldots, e_n}\) such that \(e_1 \leftjustified* \cdots \leftjustified* e_n\) and \(\galebasis\pqty{M, \leftjustified} \coloneqq \Bqty{g_1, \ldots, g_{\rankm\pqty{M}}}\) such that \(g_1 \leftjustified* \cdots \leftjustified* g_{\rankm\pqty{M}}\).
  Here \(\galebasis\pqty{M, \leftjustified}\) is the Gale basis of \(M\) with respect to \(\leftjustified\) (see \cref{sec:width_with_lambda}).
  Moreover, let \(p \coloneqq \abs{\set{j \in \Bqty{1, \ldots, \rankm\pqty{M}}}{e_{\rankm\pqty{M}} \leftjustified* g_j}}\).

  First, we show that \(\pathwidth\pqty{M} \geq p\).
  Note that \(1 \leq \rankm\pqty{M} - p + 1 \leq \rankm\pqty{M} + p \leq n\) if \(p \geq 1\).
  For any total order \(\preceq\) on \(\gset\), there exists \(i \in \Bqty{0, \ldots, n - 1}\) such that \(\abs{\ei \cap \Bqty{e_{\rankm\pqty{M} - p + 1}, \ldots, e_{\rankm\pqty{M} + p}}} = p\).
  Since \(e_j \leftjustified g_j\) for \(j \in \Bqty{1, \ldots, \rankm\pqty{M} - p}\) and \(e_j \leftjustified g_{j - p}\) for \(j \in \Bqty{\rankm\pqty{M} + 1, \ldots, \rankm\pqty{M} + p}\), we have \(\rankm_{M}\pqty{\ei} \geq \abs{\ei \cap \Bqty{e_1, \ldots, e_{\rankm\pqty{M} - p}}} + p\).
  Similarly, we have \(\rankm_{M}\pqty{\eivar} \geq \abs{\pqty{\eivar} \cap \Bqty{e_1, \ldots, e_{\rankm\pqty{M} - p}}} + p\).
  Thus we can show that \(\pathwidth\pqty{M} \geq \abs{\Bqty{e_1, \ldots, e_{\rankm\pqty{M} - p}}} + 2p - \rankm\pqty{M} = p\).

  Next, we find the upper bound on the width of the BDD/ZDD for \(M\) with respect to \(\pqty{\gset, \leftjustified}\).
  By \cref{lem:order_in_nested}, the size of \(\set{\galebasis\pqty{\minor{M}{\pqty{\ei[\leftjustified] \setminus X}}{X}, \leftjustified}}{X \subseteq \ei[\leftjustified]}\) bounds the \(i\)th width.
  Now we show that the Gale basis of \(\minor{M}{\pqty{\ei[\leftjustified] \setminus X}}{X}\) equals \(\Bqty{g_{\rankm\pqty{M} - \rankm\pqty{\minor{M}{\pqty{\ei[\leftjustified] \setminus X}}{X}} + 1}, \ldots, g_{\rankm\pqty{M}}}\) for all \(X \subseteq \ei[\leftjustified]\).
  It suffices to show that the Gale basis of \Cref{item:contraction_in_pw}~\(\contraction{M}{\Bqty{e_1}}\) is \(\Bqty{g_{\rankm\pqty{M} - \rankm\pqty{\contraction{M}{\Bqty{e_1}}} + 1}, \ldots, g_{\rankm\pqty{M}}}\) and of \cref{item:deletion_in_pw}~\(\deletion{M}{\Bqty{e_1}}\) is \(\Bqty{g_{\rankm\pqty{M} - \rankm\pqty{\deletion{M}{\Bqty{e_1}}} + 1}, \ldots, g_{\rankm\pqty{M}}}\).
  \begin{enumerate}[label=(\roman*)]
    \item\label{item:contraction_in_pw}
      Since \(e_1 \leftjustified g_1\), we have \(\rankm\pqty{\contraction{M}{\Bqty{e_1}}} = \rankm\pqty{M} - 1\).
      The same arguments as in \cref{item:contraction_in_thm} of the proof of \cref{thm:width_nested} can be applied.
    \item\label{item:deletion_in_pw}
      If \(e_1\) is a coloop, then it follows from \cref{item:contraction_in_pw} since \(\deletion{M}{\Bqty{e_1}} = \contraction{M}{\Bqty{e_1}}\).
      If \(e_1\) is neither a loop nor a coloop, then we have \(e_1 \prec^\ast g_1\).
      The set \(\galebasis\pqty{M, \leftjustified}\) is the Gale basis of \(\deletion{M}{\Bqty{e_1}}\) by the same arguments as in \cref{item:deletion_in_thm} of the proof of \cref{thm:width_nested}.
  \end{enumerate}

  Therefore, the \(i\)th width is at most
  \[
    \begin{split}
      \rankm\pqty{\deletion{M}{\ei[\leftjustified]}} - \rankm\pqty{\contraction{M}{\ei[\leftjustified]}} + 1
      &= \abs{\set{j \in \Bqty{1, \ldots, \rankm\pqty{M}}}{e_{i + 1} \leftjustified g_j}} - \pqty{\rankm\pqty{M} - \min\Bqty{i, \rankm\pqty{M}}} + 1 \\
      &= \abs{\set{j \in \Bqty{1, \ldots, \rankm\pqty{M}}}{e_{i + 1} \leftjustified g_j}} + \min\Bqty{i - \rankm\pqty{M}, 0} + 1 \text{,}
    \end{split}
  \]
  which reaches its maximum value \(p + 1\) when \(i = \min\Bqty{\rankm\pqty{M}, n - 1}\).
\end{proof}

\Cref{thm:pathwidth_partition,thm:pathwidth_nested} establish \cref{thm:intro_pathwidth}.

\subsection{Application: implementation of a rank oracle}\label{sec:rank_oracle}

Matroids are often represented as oracles in algorithms on matroids (see \cite{robinson_welsh_1980,Hausmann1981}).
This section shows how to implement matroid oracles, such as a rank oracle, for a matroid \(M\) by using the property that the 0-child (resp. 1-child) of the ZDD \(\zdd{\fun{\isets}{M}}\) which has the root with label \(e \in \fun{\gset}{M}\) represents the collection of independent sets of \(\deletion{M}{\Bqty{e}}\) (resp. \(\contraction{M}{\Bqty{e}}\)).

One of the matroid oracles is an independence oracle.
The \emph{independence oracle} for \(M\) determines whether a given set \(I \subseteq \fun{\gset}{M}\) is independent in \(M\).
Since membership queries to a ZDD can be answered in time proportional to the size of its ground set, the ZDD \(\zdd{\fun{\isets}{M}}\) can simulate this oracle in \(\order\pqty{\abs{\fun{\gset}{M}}}\) time.
This is one of the reasons why we use BDDs/ZDDs as data structures for representing matroids.

\begin{algorithm}[tbp]
  \caption{Algorithm for a rank oracle}
  \label{alg:rank_oracle}
  \begin{algorithmic}[1]
    \Require A ZDD \(\zdd{\isets}\) for a collection \(\isets\) of independent sets of a matroid \(M\) with respect to a totally ordered set \(\pqty{\fun{\gset}{M}, \preceq}\) and a set \(X \subseteq \fun{\gset}{M}\)
    \Ensure The rank \(\rankm_{M}\pqty{X}\) of \(X\)
    \Function{Rank}{$\zdd{\isets}, X$}
      \State \IfThen{\(\rootdd\pqty{\zdd{\isets}} = \top \algorithmicor X = \emptyset\)}{\Return \(0\)}
      \State \(e \coloneqq \labelroot{\zdd{\isets}}\) \Comment{The element \(e\) exists since \(\rootdd\pqty{\zdd{\isets}}\) is not terminal.}
      \State \IfThen{\(e \prec \min{X}\)}{\Return $\Call{Rank}{\childz\pqty{\zdd{\isets}}, X}$} \Comment{Calculate \(\rankm_{\deletion{M}{\Bqty{e}}}\pqty{X}\).}
      \State \algorithmicelse\ \IfThen{\(e \succ \min{X}\)}{\Return \Call{Rank}{$\zdd{\isets}, X \setminus \Bqty{\min{X}}$}} \Comment{\(\min{X}\) is a loop of \(M\).}
      \State \algorithmicelse\ \Return $\Call{Rank}{\childo\pqty{\zdd{\isets}}, X \setminus \Bqty{e}} + 1$ \Comment{Calculate \(\rankm_{\contraction{M}{\Bqty{e}}}\pqty{X \setminus \Bqty{e}} + 1\).}
    \EndFunction
  \end{algorithmic}
\end{algorithm}

Another oracle is a rank oracle.
The \emph{rank oracle} for \(M\) provides the rank \(\rankm_{M}\pqty{X}\) of a given set \(X \subseteq \fun{\gset}{M}\).
Simulating this oracle in a computer usually requires \(\order\pqty{\abs{X}}\) probes of the independence oracle for \(M\) (see \cite[Proposition 1(3)]{robinson_welsh_1980} and \cite[Proposition 2.1]{Hausmann1981}) or an \(\order\pqty{2^{\abs{\fun{\gset}{M}}}}\) space preprocessing.
However, the ZDD \(\zdd{\fun{\isets}{M}}\) can simulate the rank oracle in \(\order\pqty{\abs{\fun{\gset}{M}}}\) time as shown in \cref{alg:rank_oracle}.
This algorithm is based on the following (see \cite[Proposition 3.1.6]{10.1093/acprof:oso/9780198566946.001.0001}):
\[
  \rankm_{M}\pqty{X} = \begin{cases}
    \rankm_{\deletion{M}{\Bqty{e}}}\pqty{X} & \text{if \(e \notin X\),} \\
    \rankm_{\contraction{M}{\Bqty{e}}}\pqty{X \setminus\Bqty{e}} + \rankm_{M}\pqty{\Bqty{e}} & \text{if \(e \in X\),} \\
  \end{cases}
\]
for all \(e \in \fun{\gset}{M}\), which allows us to reduce the calculation to a case with a smaller ground set.
Other basic oracles for \(M\) can be simulated in \(\order\pqty{\abs{\fun{\gset}{M}}}\) or \(\order\pqty{\abs{\fun{\gset}{M}}^2}\) time by using this rank oracle.

\section*{Acknowledgments}
\addcontentsline{toc}{section}{Acknowledgments}

The second author was supported by JSPS KAKENHI Grant Number JP22K17854.
The third author was supported by JSPS KAKENHI Grant Numbers JP20H00605, JP20H05964.

\bibliographystyle{abbrvnat}
\bibliography{references}
\addcontentsline{toc}{section}{References}

\end{document}